\documentclass[reqno]{amsart}
\usepackage{amsmath, amsthm, amssymb, amstext}

\usepackage[left=3.5cm,right=3.5cm,top=2cm,bottom=2cm,includeheadfoot]{geometry}
\usepackage{hyperref,xcolor}
\hypersetup{
pdfborder={0 0 0},
colorlinks,
}
\usepackage{enumitem}
\newtheorem{theorem}{Theorem}
\newtheorem{remark}[theorem]{Remark}
\newtheorem{lemma}[theorem]{Lemma}
\newtheorem{proposition}[theorem]{Proposition}
\newtheorem{corollary}[theorem]{Corollary}
\newtheorem{definition}[theorem]{Definition}
\newtheorem{example}[theorem]{Example}

\newcommand{\wto}{\overset{w}{\longrightarrow}}

\DeclareMathOperator*{\divergenz}{div}              %
\DeclareMathOperator*{\Ss}{S}

\newcommand{\N}{\mathbb{N}}

\newcommand{\R}{\mathbb{R}}

\newcommand{\Lp}[1]{L^{#1}(\Omega)}

\newcommand{\Lprand}[1]{L^{#1}(\partial\Omega)}
\newcommand{\Wp}[1]{W^{1,#1}(\Omega)}

\newcommand{\into}{\int_{\Omega}}

\newcommand{\Linf}{L^{\infty}(\Omega)}

\renewcommand{\l}{\left}
\renewcommand{\r}{\right}
\numberwithin{theorem}{section}
\numberwithin{equation}{section}

\newcommand{\WH}{W^{1, \mathcal{H}}(\Omega)}

\title[Double phase obstacle problems with multivalued convection]{Double phase obstacle problems with multivalued convection and mixed boundary value conditions}

\author[S. Zeng]{Shengda Zeng}
\address[S. Zeng]{Guangxi Colleges and Universities Key Laboratory of Complex System Optimization and Big Data Processing, Yulin Normal University, Yulin 537000, Guangxi, P.R. China, \&Department of Mathematics, Nanjing University, Nanjing, Jiangsu, 210093, P.R. China, \& Jagiellonian University in Krakow,
Faculty of Mathematics and Computer Science, ul. Lojasiewicza 6, 30-348 Krakow, Poland}
\email{zengshengda@163.com}

\author[V.D. R\u{a}dulescu]{Vicen\c{t}iu D.\,R\u{a}dulescu}
\address[V.D. R\u{a}dulescu]{Faculty of Applied Mathematics, AGH University of Science and Technology, 30-059 Krak\'ow, Poland \& Department of Mathematics, University of Craiova, 200585 Craiova, Romania \& China-Romania Research Center in Applied Mathematics, Craiova, Romania}
\email[Corresponding author]{radulescu@inf.ucv.ro}

\author[P.\,Winkert]{Patrick Winkert}
\address[P.\,Winkert]{Technische Universit\"{a}t Berlin, Institut f\"{u}r Mathematik, Stra\ss e des 17.\,Juni 136, 10623 Berlin, Germany}
\email{winkert@math.tu-berlin.de}

\subjclass[2010]{35J20, 35J25, 35J60}
\keywords{Clarke's generalized subgradient, Double phase problem, Kuratowski upper limit, Moreau-Yosida approximation, multivalued convection term, obstacle problem, surjectivity theorem}

\begin{document}

\begin{abstract}
	In this paper, we consider a mixed boundary value problem with a double phase partial differential operator, an obstacle effect and a multivalued reaction convection term. Under very general assumptions, an existence theorem for the mixed boundary value problem under consideration is proved by using a surjectivity theorem for multivalued pseudomonotone operators together with the approximation method of Moreau-Yosida. Then, we introduce a family of the approximating problems without constraints  corresponding to the mixed boundary value problem. Denoting by $\mathcal S$ the solution set of the mixed boundary value problem and by $\mathcal S_n$ the solution sets of the approximating problems, we establish the following convergence relation
	\begin{align*}
		\emptyset\neq w\text{-}\limsup\limits_{n\to\infty}{\mathcal S}_n=s\text{-}\limsup\limits_{n\to\infty}{\mathcal S}_n\subset \mathcal S,
	\end{align*}
	where $w$-$\limsup_{n\to\infty}\mathcal S_n$ and $s$-$\limsup_{n\to\infty}\mathcal S_n$ stand for the weak and the strong Kuratowski upper limit of $\mathcal S_n$, respectively.
\end{abstract}

\maketitle

\section{Historical comments and statement of the problem}

The study of obstacle problem goes back to the pioneering contributions of J.-L.~Lions \cite{lions} who studied the following simple, beautiful and deep problem: find the equilibrium position $u=u(x)$  of an elastic membrane restricted to lie above a given obstacle $\psi=\psi(x)$, where  $x\in\Omega$ and $\Omega\subset \R^2$ is a bounded domain with smooth boundary. This equilibrium position is the unique minimizer of the Dirichlet energy functional, that is,
$$\min_{v\in K}\int_\Omega |\nabla v|^2\,\mathrm{d} x,$$
where $K$ is a suitable convex set of functions greater or equal to $\psi$. This problem turns out to be equivalent to the following variational inequality
$$u\in K\ \text{and}\ \int_\Omega \nabla u\cdot\nabla (v-u)\,\mathrm{d} x\geq 0\quad \text{ for all } v\in K,$$
which can be seen as a system of Euler inequalities for the corresponding minimization problem.
In the region $[v>\psi]$ where the membrane is above the obstacle, the solution $u$ solves an elliptic equation (say, $\Delta u=0$), while in the other region the membrane coincides with the obstacle (namely, $u=\psi$).
The region $[v=\psi]$ is known as the contact set and the interface that separates the two regions is the free boundary.
In such a way, the obstacle problem has been introduced for a membrane and for a plate and it is the simplest unilateral problem from the classical elasticity theory, as well as for the nonparametric minimal and capillary surfaces.
Various classes of obstacle problems arise naturally when describing
phenomena in the real world. Many of these models (fluid filtration through
porous medium, osmosis, optimal stopping, heat control, etc.) are described
in the monographs by Duvaut \& Lions \cite{duvaut} and Rodrigues \cite{rodri}.

Let  $\Omega\subset \R^N$, $N>2$,  be a bounded domain such that its boundary $\Gamma:=\partial \Omega$ is Lipschitz continuous and $\Gamma$ is divided into two disjoint measurable parts $\Gamma_1$  and $\Gamma_2$ with meas$(\Gamma_1)>0$. Let $1<p<q<N$ and let $\nu$ be the outward unit normal at the boundary $\Gamma$. Given a bounded function $\mu\colon \Omega\to [0,+\infty)$, a multivalued operator $f\colon \Omega\times \R\times \R^N\to 2^{\R}$,  an obstacle function $\Phi\colon \Omega \to \R_+$ and a function $j\colon \Gamma_2\times \R\to \R$, we consider the following elliptic obstacle inclusion problem with a double phase differential operator, a multivalued convection term and mixed boundary conditions, where the boundary conditions are composed of a homogeneous Dirichlet boundary condition and a multivalued boundary condition which is described by Clarke's generalized subgradient:
\begin{equation}\label{eqn1}
	\begin{aligned}
		D_\mu(u) +|u|^{p-2}u+\mu(x)|u|^{q-2}u&\in f(x,u,\nabla u)\quad && \text{in } \Omega,\\
		u  &\leq \Phi\quad && \text{in } \Omega,\\
		u  &= 0 &&\text{on } \Gamma_1,\\
		\frac{\partial u}{\partial \nu_\mu}&\in -\partial j(x,u)&&\text{on }\Gamma_2,
	\end{aligned}
\end{equation}
where
\begin{align*}
	D_\mu(u):=-\divergenz\left(|\nabla u|^{p-2}\nabla u+\mu(x) |\nabla u|^{q-2}\nabla u\right)
\end{align*}
and
\begin{align*}
	\frac{\partial u}{\partial \nu_\mu}:=\left(|\nabla u|^{p-2}\nabla u+\mu(x) |\nabla u|^{q-2}\nabla u\right) \cdot \nu.
\end{align*}

It should be mentioned that in our setting the part $\Gamma_2$ can be empty. In this case, problem \eqref{eqn1} reduces to the Dirichlet double phase problem
\begin{equation}\label{eqn1.0}
	\begin{aligned}
		D_\mu(u)+|u|^{p-2}u+\mu(x)|u|^{q-2}u &\in f(x,u,\nabla u)\quad && \text{in } \Omega,\\
		u  &\leq \Phi\quad && \text{in } \Omega,\\
		u  &= 0 &&\text{on } \Gamma,
	\end{aligned}
\end{equation}
which has recently been introduced and studied by Zeng, Gasi\'nski, Winkert \& Bai~\cite{Zeng-Gasinski-Winkert-Bai-2021}.

In general, the novelty of the present work is the fact that several interesting and challenging phenomena are considered in one problem. To be more precise, problem \eqref{eqn1} contains the following effects:
\begin{enumerate}
	\item[(i)]
		a double phase partial differential operator;
	\item[(ii)]
		a multivalued convection term;
	\item[(iii)]
		an obstacle restriction;
	\item[(iv)]
		a multivalued boundary condition, which is formulated by Clarke's generalized subdifferential operator for locally Lipschitz functions.
\end{enumerate}

Problem \eqref{eqn1} appears naturally when considering optimal stopping problems for L\'evy processes with jumps, which arise for example as option pricing models in mathematical finance.
The main contribution of this paper is twofold. The first objective is to explore the nonemptiness, boundedness and closedness of the solution set $\mathcal{S}$ to problem \eqref{eqn1}. Our method is based on a surjectivity theorem for multivalued pseudomonotone operators and the approximation method of Moreau-Yosida. Since the obstacle effect leads to various difficulties in obtaining the exact and numerical solutions, some appropriate and useful approximating methods have been introduced and developed to overcome the obstacle constraints.  Based on this, the second part of this paper is aimed to consider a family of approximating problems corresponding to \eqref{eqn1}  without constraints and to establish a critical convergence theorem which indicates that the solution set of the obstacle problem can be approximated by the solution sets of approximating problems, denoted by $\{\mathcal S_n\}$, in the sense of Kuratowski. More precisely, we are going to show that
\begin{align*}
	\emptyset\neq w\text{-}\limsup\limits_{n\to\infty}{\mathcal S}_n=s\text{-}\limsup\limits_{n\to\infty}{\mathcal S}_n\subset \mathcal S,
\end{align*}
where $w$-$\limsup_{n\to\infty}\mathcal S_n$ is the weak Kuratowski upper limit of $\mathcal S_n$ and $s$-$\limsup_{n\to\infty}\mathcal S_n$ stands for the strong Kuratowski upper limit of $\mathcal S_n$.

In 1986, Zhikov \cite{Zhikov-1986} initially introduced a nonlinear and nonhomogeneous integral functional
\begin{align}\label{integral_minimizer}
   u \mapsto \int \left(|\nabla  u|^p+\mu(x)|\nabla  u|^q\right)\,\mathrm{d}x
\end{align}
to investigate the mechanics problems for strongly anisotropic materials. It is not difficult to see that the corresponding differential form of \eqref{integral_minimizer} is written by
\begin{eqnarray}\label{double_phase_operator}
u\mapsto-\divergenz\left(|\nabla u|^{p-2}\nabla u+\mu(x) |\nabla u|^{q-2}\nabla u\right).
\end{eqnarray}
On the one hand, from the physical point of view, the integral functional \eqref{integral_minimizer}  describes exactly the  phenomenon that the energy density changes its ellipticity and growth properties according to the point in the domain. On the other hand, from the mathematical point of view, the behavior of the integral functional \eqref{integral_minimizer} depends on the values of the weight function $\mu(\cdot)$. More precisely, on the set $\{x\in \Omega: \mu(x)=0\}$ it will be controlled by the gradient of order $p$ and in the case $\{x\in \Omega: \mu(x) \neq 0\}$ it is the gradient of order $q$. This is the essential reason why we call \eqref{double_phase_operator} as double phase partial differential operator.

Double phase differential operators and corresponding energy functionals  interpret various comprehensive natural phenomena, and model several problems in Mechanics, Physics and Engineering Sciences. For example, in the elasticity theory, the modulating coefficient $\mu(\cdot)$ dictates the geometry of composites made of two different materials with distinct power hardening exponents $q$ and $p$, see~\cite{Zhikov-2011}.  Thereafter, this topic attracted increasing attention and witnessed plenty of significant results in various aspects. For instance, Gasi\'nski \& Winkert~\cite{Gasinski-Winkert-2020b} applied the theory of pseudomonotone operators to study the existence and uniqueness of solution for a quasilinear elliptic equation with double phase phenomena and a reaction term depending on the gradient. Under the assumption that the reaction is superlinear but without satisfying the Ambrosetti-Rabinowitz condition, Papageorgiou, R\u adulescu \& Repov\v s~\cite{Papageorgiou-Radulescu-Repovs-2020} considered a double phase Robin problem with a Carath\'eodory nonlinearity and proved an existence theorem as well as a multiplicity theorem using Morse theoretic tools and the notion of homological local linking. By establishing a weighted inequality for a Baouendi-Grushin operator and a related compactness property, Bahrouni, R\u{a}dulescu \& Repov\v{s}~\cite{Bahrouni-Radulescu-Repovs-2019} obtained the existence of stationary waves under arbitrary perturbations of the reaction for a class of double phase transonic flow problems with variable growth.

Apart from their obvious importance in the theory of partial differential equations, obstacle problems have a
natural theoretical interest in stochastic control. Additionally, they can be found in physics, biology, and mathematical finance.
One of the most well-known financial challenges is establishing the arbitrage-free price of American-style options.
Concerning the mathematical analysis of obstacle problems, we refer to
the recent contribution of  Zeng, Bai, Gasi\'nski \& Winkert~\cite{Zeng-Bai-Gasinski-Winkert-2020} who applied a surjectivity theorem for multivalued mappings, Kluge's fixed point principle and tools from nonsmooth analysis to explore the existence of weak solutions for a new kind of implicit obstacle problems driven by a double phase partial differential operator and a multivalued term which is described by Clarke's generalized gradient. 
We also refer to Bertoin \cite{bertoin} for L\'evy processes with jumps, which
 arise in the description of various phenomena
in the applied sciences, such as plasma physics, flame propagation, free
boundary obstacle problems,  or phase transitions in the Gamma convergence framework.
For further results concerning single-valued equations involving double phase operators or multivalued equations with or without double phase operators  we refer to the works of Alves, Garain \& R\u adulescu~\cite{Alves-Garain-Radulescu-2021}, Ambrosio \& R\u adulescu \cite{Ambrosio-Radulescu-2020}, Bahrouni, R\u{a}dulescu \& Winkert \cite{Bahrouni-Radulescu-Winkert-2020,Bahrouni-Radulescu-Winkert-2019}, Baroni, Colombo \& Mingione~\cite{Baroni-Colombo-Mingione-2015,Baroni-Colombo-Mingione-2016,Baroni-Colombo-Mingione-2018}, Cencelj, R\u{a}dulescu \& Repov\v{s} \cite{Cencelj-Radulescu-Repovs-2018}, Colasuonno \& Squassina \cite{Colasuonno-Squassina-2016}, Colombo \& Mingione \cite{Colombo-Mingione-2015a,Colombo-Mingione-2015b}, Farkas \& Winkert~\cite{Farkas-Winkert-2021},  Gasi\'nski \& Papageorgiou \cite{Gasinski-Papageorgiou-2019, Gasinski-Papageorgiou-2017,Gasinski-Papageorgiou-2005}, Gasi\'n\-ski \& Winkert \cite{Gasinski-Winkert-2020a, Gasinski-Winkert-2020b, Gasinski-Winkert-2021}, Liu \& Dai \cite{Liu-Dai-2018}, Marino \& Winkert \cite{Marino-Winkert-2020}, Papageorgiou, R\u{a}dulescu \& Repov\v{s} \cite{Papageorgiou-Radulescu-Repovs-2019, Papageorgiou-Radulescu-Repovs-2018}, Papageorgiou, Vetro \& Vetro \cite{Papageorgiou-Vetro-Vetro-2020a,Papageorgiou-Vetro-Vetro-2021,Papageorgiou-Vetro-Vetro-2020b}, Perera \& Squassina \cite{Perera-Squassina-2019}, R\u{a}dulescu \cite{Radulescu-2019}, Vetro \cite{Vetro-2021}, Vetro \& Vetro \cite{Vetro-Vetro-2020}, Zhang \& R\u{a}dulescu \cite{Zhang-Radulescu-2018},
see also the references therein. 
Basic analytic tools used in this paper can be found in the monographs \cite{galewski, Papageorgiou-Winkert-2018}.
Finally, we mention the  overview article of Mingione \& R\u{a}dulescu \cite{Mingione-Radulescu-2021}  about recent developments for problems with nonstandard growth and nonuniform ellipticity.

The paper is organized as follows. Section \ref{Preliminaries} is devoted to recall some useful and important preliminaries such as the Musielak-Orlicz spaces $\Lp{\mathcal{H}}$ and its corresponding Sobolev spaces $\Wp{\mathcal{H}}$, the definition of  Kuratowski lower and upper limit, as well as the Moreau-Yosida approximation to proper, convex and lower semicontinuous functions, respectively. In Section \ref{Section3}, we apply a surjectivity result for multivalued pseudomonotone operators combined with the Moreau-Yosida approximation to examine the nonemptiness, boundedness and closedness of solution set to problem \eqref{eqn1}, see Theorem \ref{theorems3.3}. Finally, in Section \ref{Section4}, a family of the approximating problems without constraints corresponding to problem \eqref{eqn1} is introduced and an impressive convergence theorem is obtained, which reveals the essential relations between the sets $\mathcal S$, $w$-$\limsup_{n\to\infty}\mathcal S_n$  and $s$-$\limsup_{n\to\infty}\mathcal S_n$, see Theorem~\ref{main_theorem}.

\section{Mathematical background}\label{Preliminaries}

Let $\Omega$ be a bounded domain in $\R^N$ and let $1\leq r<\infty$. For any subset $D$ of $\overline \Omega$, in what follows, we denote by $L^r(D):=L^r(D;\R)$ and $L^r(D;\R^N)$ the usual Lebesgue spaces endowed with the norm $\|\cdot\|_{r,D}$, that is,
\begin{align*}
	\|u\|_{r,D}:=\left(\int_D |u|^r\,\mathrm{d}x\right)^\frac{1}{r}\quad \text{for all }u\in L^r(D).
\end{align*}
We set $L^r(D)_+:=\{u\in L^r(D)\,:\,u(x)\ge 0\text{ for a.\,a.\,}x\in D\}$.
Moreover, $W^{1,r}(\Omega)$ stands for the Sobolev space endowed with the norm $\|\cdot\|_{1,r,\Omega}$, namely,
\begin{align*}
	\|u\|_{1,r,\Omega}:=\|u\|_{r,\Omega}+\|\nabla u\|_{r,\Omega}\quad \text{for all }u\in W^{1,r}(\Omega).
\end{align*}

For any $1<r<\infty$ we denote by $r'$ the conjugate exponent of $r$, that is, $\frac{1}{r}+\frac{1}{r'}=1$. In the sequel, we denote by $r^*$ and $r_*$  the critical exponents to $r$ in the domain and on the boundary, respectively, given by
\begin{align}\label{defp1*}
		r^*=
		\begin{cases}
			\frac{Nr}{N-r} &\text{ if }r<N,\\
			+\infty &\text{ if }r\ge N,
		\end{cases}
		\quad\text{and}\quad
		r_*=
		\begin{cases}
			\frac{(N-1)r}{N-r} &\text{ if }r<N,\\
			+\infty &\text{ if }r\ge N,
		\end{cases}
\end{align}
respectively.

Consider the $r$-Laplacian eigenvalue problem with Steklov boundary condition given by
\begin{equation}\label{Steklovbc}
	\begin{aligned}
	-\Delta_ru&=-|u|^{r-2}u \quad && \text{in } \Omega,\\
	|\nabla u|^{r-2}\nabla u\cdot\nu&=\lambda|u|^{r-2}u  &&\text{on } \Gamma,
	\end{aligned}
\end{equation}
for $1<r<\infty$. From L{\^e} \cite{Le-2006} we can see that problem \eqref{Steklovbc} admits a smallest eigenvalue $\lambda_{1,r}^S>0$ which is isolated and simple. In addition, the  smallest eigenvalue $\lambda_{1,r}^S>0$  satisfies the following variational equality:
\begin{equation}\label{Steklovbc2}
	\lambda_{1,r}^S=\inf_{u\in W^{1,r}(\Omega)\setminus\{0\}}\frac{\|\nabla u\|_{r,\Omega}^r+\|u\|_{r,\Omega}^r}{\|u\|_{r,\Gamma}^r}.
\end{equation}

For problem \eqref{eqn1}, in the whole paper, we assume that the weight function $\mu$ and powers $p$, $q$ satisfy the following conditions:
\begin{align}\label{conditions-p-q-mu}
	1<p<N,\quad p<q<p^*\quad \text{and}\quad 0\leq \mu(\cdot)\in \Linf.
\end{align}

Let us introduce the nonlinear function $\mathcal H\colon \Omega\times\R_+\to \R_+$ given by
\begin{align*}
	\mathcal H(x,t):=t^p+\mu(x)t^q \quad \text{for all }(x,t)\in \Omega\times \R_+,
\end{align*}
where $\R_+=[0,\infty)$. By virtue of the definition of $\mathcal{H}$, we are now in a position to recall the well-known Musielak-Orlicz function space denoted by $L^\mathcal H(\Omega)$, that is,
\begin{align*}
	L^\mathcal{H}(\Omega):=\left \{u ~ \Big | ~ u: \Omega \to \R \text{ is measurable and } \rho_{\mathcal{H}}(u):=\into \mathcal{H}(x,|u|)\,\mathrm{d}x< +\infty \right \}.
\end{align*}
It is obvious that $L^\mathcal H(\Omega)$ endowed with the Luxemburg norm
\begin{align*}
  \|u\|_{\mathcal{H}} = \inf \left \{ \tau >0 \ \Big  | \ \rho_{\mathcal{H}}\left(\frac{u}{\tau}\right) \leq 1  \right \}\quad \text{for all }u\in L^\mathcal H(\Omega)
\end{align*}
is a reflexive Banach space, see Colasuonno \& Squassina \cite[Proposition 2.14]{Colasuonno-Squassina-2016}. Moreover, let us consider the seminormed function space defined by
\begin{align*}
	L^q_\mu(\Omega)=\left \{u ~ \Big | ~ u: \Omega \to \R \text{ is measurable and } \into \mu(x) |u|^q \,\mathrm{d}x< +\infty \right \},
\end{align*}
which is equipped  with the seminorm $\|\cdot\|_{q,\Omega,\mu}$ given by
\begin{align*}
	\|u\|_{q,\Omega,\mu} = \left(\into \mu(x) |u(x)|^q \,\mathrm{d}x \right)^{\frac{1}{q}}\quad\text{for all }u\in  L^q_\mu(\Omega).
\end{align*}

By $W^{1,\mathcal{H}}(\Omega)$ we denote the corresponding Musielak-Orlicz Sobolev space defined by
\begin{align*}
	W^{1,\mathcal{H}}(\Omega)= \left \{u \in L^\mathcal{H}(\Omega) \,\mid\, |\nabla u| \in L^{\mathcal{H}}(\Omega) \right\}
\end{align*}
equipped with the norm
\begin{align*}
	\|u\|_{1,\mathcal{H}}= \|\nabla u \|_{\mathcal{H}}+\|u\|_{\mathcal{H}}\quad \text{for all }u\in W^{1,\mathcal H}(\Omega),
\end{align*}
where $\|\nabla u\|_\mathcal{H}=\|\,|\nabla u|\,\|_{\mathcal{H}}$. Since problem \eqref{eqn1} is a mixed boundary value problem with double phase partial differential operator, we introduce a closed subspace $V$ of $W^{1,\mathcal{H}}(\Omega)$ defined by
\begin{align*}
	V:=\{u\in W^{1,\mathcal{H}}(\Omega)\,\mid\,u=0\text{ on $\Gamma_1$}\},
\end{align*}
which is also a reflexive Banach space.  For the sake of convenience, in what follows, we denote by $\|\cdot\|_V$  the norm of $V$, that is, $\|u\|_V=  \|u\|_{1,\mathcal{H}}$ for all $u\in V$, and by $V^*$ we denote the dual space of $V$.

Next, we collect some useful embedding results for the spaces $L^\mathcal H(\Omega)$ and $W^{1,\mathcal H}(\Omega)$. We refer, for example, to Gasi\'nski \& Winkert \cite[Proposition 2.2]{Gasinski-Winkert-2021} or Crespo-Blanco, Gasi\'nski, Harjulehto \& Winkert \cite[Proposition 2.17]{Crespo-Blanco-Gasinski-Winkert-2021}.

\begin{proposition}\label{proposition_embeddings}
	Let \eqref{conditions-p-q-mu} be satisfied and let $p^*$ as well as $p_*$ be the critical exponents to $p$ as given in \eqref{defp1*} for $r=p$. Then the following embeddings hold:
	\begin{enumerate}
		\item[\textnormal{(i)}]
		$\Lp{\mathcal{H}} \hookrightarrow \Lp{r}$ and $\WH\hookrightarrow \Wp{r}$ are continuous for all $r\in [1,p]$;
		\item[\textnormal{(ii)}]
		$\WH \hookrightarrow \Lp{r}$ is continuous for all $r \in [1,p^*]$ and compact for all $r \in [1,p^*)$;
		\item[\textnormal{(iii)}]
		$\WH \hookrightarrow \Lprand{r}$ is continuous for all $r \in [1,p_*]$ and compact for all $r \in [1,p_*)$;
		\item[\textnormal{(iv)}]
		$\Lp{\mathcal{H}} \hookrightarrow L^q_\mu(\Omega)$ is continuous;
		\item[\textnormal{(v)}]
		$\Lp{q}\hookrightarrow\Lp{\mathcal{H}} $ is continuous.
	\end{enumerate}
\end{proposition}

\begin{remark}
It is obvious that if we replace the space $W^{1,\mathcal H}(\Omega)$ by $V$, then the embeddings of \textnormal{(ii)} and \textnormal{(iii)} in Proposition \ref{proposition_embeddings} hold as well.
\end{remark}

From Liu \& Dai \cite[Proposition 2.1]{Liu-Dai-2018} we have the following proposition.
\begin{proposition}
	Let \eqref{conditions-p-q-mu} be satisfied and let $y\in \Lp{\mathcal{H}}$. Then the following hold:
	\begin{enumerate}
		\item[{\rm(i)}]
			if $y\neq 0$, then $\|y\|_\mathcal H=\lambda$ if and only if $\rho_\mathcal H\left(\frac{y}{\lambda}\right)=1$;
		\item[{\rm(ii)}]
			$\|y\|_\mathcal H<1$ (resp. $>1$ and $=1$) if and only if $\rho_\mathcal H(y)<1$ (resp. $>1$ and $=1$);
		\item[{\rm(iii)}]
			if $\|y\|_\mathcal H<1$, then $\|y\|_{\mathcal H}^q\le \rho_\mathcal H(y)\le \|y\|_\mathcal H^p$;
		\item[{\rm(iv)}]
			if $\|y\|_\mathcal H>1$, then $\|y\|_{\mathcal H}^p\le \rho_\mathcal H(y)\le \|y\|_\mathcal H^q$;
		\item[{\rm(v)}]
			$\|y\|_\mathcal H\to 0$ if and only if $\rho_\mathcal H(y)\to 0$;
		\item[{\rm(vi)}]
			$\|y\|_\mathcal H\to +\infty$ if and only if $ \rho_\mathcal H(y)\to +\infty$.
	\end{enumerate}
\end{proposition}

Throughout the paper the symbols "$\wto$" and "$\to$" stand for the weak and the strong convergence, respectively. For a Banach space $(X,\|\cdot\|_X)$ we denote its dual space by $X^*$ and by $\langle\cdot,\cdot\rangle_{X^*\times X}$ the duality pairing between $X^*$ and $X$.

Furthermore, we consider the nonlinear operator $A\colon V\to V^*$ defined by
\begin{equation}\label{defA}
	\langle A(u),v\rangle :=\into \left(|\nabla u|^{p-2}\nabla u+\mu(x)|\nabla u|^{q-2}\nabla u \right)\cdot\nabla v \,\mathrm{d}x+\into \left(|u|^{p-2} u+\mu(x)| u|^{q-2} u \right) v \,\mathrm{d}x
\end{equation}
for $u,v\in V$, where $\langle \cdot,\cdot\rangle$ stands for the duality pairing between $V$ and its dual space $V^*$. The properties of the operator $A\colon V\to V^*$ can be summarized as follows, where its detailed proof could be found in  Liu \& Dai \cite{Liu-Dai-2018} or Crespo-Blanco, Gasi\'nski, Harjulehto \& Winkert \cite[Proposition 3.5]{Crespo-Blanco-Gasinski-Winkert-2021}.

\begin{proposition}\label{prop1}
	The operator $A$ defined by \eqref{defA} is bounded, continuous, monotone (hence maximal monotone) and of type $(\Ss_+)$, that is,
	\begin{align*}
		u_n\wto u \quad \text{in }V \quad\text{and}\quad  \limsup_{n\to\infty}\langle Au_n,u_n-u\rangle\le 0,
	\end{align*}
	imply $u_n\to u$ in $V$.
\end{proposition}

Given a real Banach space $(E,\|\cdot\|_E)$, we say that a function $j\colon E\to \R$ is locally Lipschitz at $x\in E$, if there is a neighborhood $O(x)$ of $x$ and a constant $L_x>0$ such that
\begin{align*}
	|j(y)-j(z)|\leq L_x\|y-z\|_E \quad \text{for all } y, z\in O(x).
\end{align*}
We denote by
\begin{align*}
	j^\circ(x;y): = \limsup \limits_{z\to x,\, \lambda\downarrow 0 } \frac{j(z+\lambda y)-j(z)}{\lambda},
\end{align*}
the generalized directional derivative of $j$ at the point $x$ in the direction $y$
and $\partial j\colon E\to 2^{E^*}$ defined by
\begin{align*}
	\partial j(x): =\left \{\, \xi\in E^{*} \, \mid \, j^\circ (x; y)\geq \langle\xi, y\rangle_{E^*\times E} \ \text{ for all } y \in    E  \right \} \quad \text{for all } x\in E
\end{align*}
stands for the generalized gradient of $j$ at $x$ in the sense of Clarke.

We next collect some properties for the generalized gradient and generalized directional derivative of a locally Lipschitz function, see for example, Mig\'{o}rski, Ochal \& Sofonea \cite[Proposition 3.23]{Migorski-Ochal-Sofonea-2013}.

\begin{proposition}
	Let $j\colon E \to \R$ be locally Lipschitz with Lipschitz constant $L_{x}>0$ at $x\in E$. Then we have the following:
	\begin{enumerate}
		\item[{\rm(i)}]
			The function $y\mapsto j^\circ(x;y)$ is positively    homogeneous, subadditive, and satisfies
			\begin{align*}
				|j^\circ(x;y)|\leq L_{x}\|y\|_E \quad \text{for all }y\in E.
			\end{align*}
		\item[{\rm(ii)}]
			The function $(x,y)\mapsto j^\circ(x;y)$ is upper semicontinuous.
		\item[{\rm(iii)}]
			For each $x\in E$, $\partial j(x)$ is a nonempty, convex, and weakly$^*$ compact subset of $E^*$ with $ \|\xi\|_{E^{*}}\leq L_{x}$ for all $\xi\in\partial j(x)$.
		\item[{\rm(iv)}]
			$j^\circ(x;y) = \max\left\{\langle\xi,y\rangle_{E^*\times E} \mid \xi\in\partial j(x)\right\}$ for all $y\in E$.
		\item[{\rm(v)}]
			The multivalued function $E\ni x\mapsto \partial j(x)\subset E^*$ is upper semicontinuous from $E$ into w$^*$-$E^*$.
    \end{enumerate}
\end{proposition}

Next we recall the following definition, see, for example, Papageorgiou \& Winkert \cite[Definition 6.7.4]{Papageorgiou-Winkert-2018}.

\begin{definition}
	Let $(X,\tau)$ be a Hausdorff topological space and let $\{A_n\}_{n\in\N}\subset 2^X$ be a sequence of sets. We define the $\tau$-Kuratowski lower limit of the sets $A_n$ by
	\begin{align*}
		\tau\text{-}\liminf\limits_{n\to\infty}A_n:=\left\{\, x\in X \mid x=\tau\text{-}\lim\limits_{n\to\infty}x_n,\,x_n\in A_n \ \text{\rm for all } \ n\ge 1\, \right\},
	\end{align*}
	and the $\tau$-Kuratowski upper limit of the sets $A_n$
	\begin{align*}
		\tau\text{-}\limsup\limits_{n\to\infty}A_n:=\left\{\, x\in X \mid 	x=\tau\text{-}\lim\limits_{k\to\infty}x_{n_k}, \,
		x_{n_k}\in A_{n_k}, \, n_1<n_2<\ldots<n_k<\ldots \, \right\}.
	\end{align*}
	If
	\begin{align*}
		A=\tau\text{-}\liminf\limits_{n\to\infty}A_n=\tau\text{-}\limsup\limits_{n\to\infty}A_n,
	\end{align*}
	then $A$ is called $\tau$-Kuratowski limit of the sets $\{A_n\}_{n\in\N}$.
\end{definition}

We end this section by recalling the definition of  the Moreau-Yosida approximation for proper, convex and lower semicontinuous functions including its properties, see Papageorgiou, Kyritsi \& Yiallourou \cite[Definition 3.2.48 and Proposition 3.2.50]{Papageorgiou-Kyritsi-Yiallourou-2009}.

\begin{lemma}\label{lemma2}
	Let $X$ be a Banach space and $\varphi\colon X\to\overline{\R}:=\R\cup\{+\infty\}$ be a proper, convex and lower semicontinuous function. Then, for $\varepsilon>0$, the Moreau-Yosida approximation $\varphi_\varepsilon\colon X\to \R$ of $\varphi$ defined by
	\begin{equation*}
		\varphi_\varepsilon(u)=\inf \bigg\{\frac{\|u-v\|^2_X}{2\varepsilon}+\varphi(v)\,\mid\, v\in X\bigg\} \quad \text{ for all }\ u\in X
	\end{equation*}
	satisfies the following properties:
	\begin{enumerate}
		\item[{\rm (i)}]
			$\varphi_\varepsilon$ is convex, lower semicontinuous and G\^ateaux differentiable.
		\item[{\rm (ii)}]
			The differential operator $\varphi_\varepsilon'\colon X\to X^*$ is bounded, monotone and demicontinuous.
		\item[{\rm (iii)}]
			If $u_\varepsilon \wto u$ in $X$, then we have
			\begin{align*}
				\limsup_{\varepsilon\to 0}\varphi_\varepsilon(v)&\le \varphi(v) \quad \text{ for all }\ v\in X,\\
				\varphi(u)&\le \liminf_{\varepsilon\to 0}\varphi_\varepsilon(u_\varepsilon).
			\end{align*}
	\end{enumerate}
\end{lemma}

\section{Existence result}\label{Section3}

The main objective of this section is to investigate the nonemptiness, boundedness and closedness of the solution set to problem \eqref{eqn1}. To this end, we impose the following hypotheses on the data of problem \eqref{eqn1}:

\begin{enumerate}
	\item[H($f$):]
	The multivalued convection mapping $f\colon \Omega \times \R\times \R^N\to 2^\R$ has nonempty, compact and convex values such that
	\begin{enumerate}
		\item[(i)]
			the multivalued mapping $x\mapsto f(x,s,\xi)$ has a measurable selection for all $(s,\xi)\in\R\times \R^N$;
		\item[(ii)]
			the multivalued mapping $(s,\xi)\mapsto f(x,s,\xi)$ is upper semicontinuous for a.\,a.\,$x\in\Omega$;
		\item[(iii)]
			there exist $c_f\in L^{q_1'}(\Omega)_+$ and $a_f,b_f\ge 0$ such that
			\begin{align*}
				|\eta|\le a_f|\xi|^{\frac{p}{q_1'}}+b_f|s|^{q_1-1}+c_f(x)
			\end{align*}
			for all $\eta\in f(x,s,\xi)$, for all $s\in \R$, for all $\xi\in\R^N$ and for a.\,a.\,$x\in \Omega$, where $1<q_1<p^*$ with the critical exponent $p^*$ in the domain $\Omega$ given in \eqref{defp1*} for $r=p$;
		\item[(iv)]
			there exist $d_f\in L^1_+(\Omega)$, $e_f,g_f\ge 0$  and $\theta_2,\theta_3\in[1,p]$ such that
			\begin{align*}
		\max\left\{e_f\delta(\theta_2))+c_j\left(\lambda_{1,p}^{S}\right)^{-1}, g_f\delta(\theta_3)+c_j\left(\lambda_{1,p}^{S}\right)^{-1}\delta(\theta_1)\right\}<1,
			\end{align*}
			and
			\begin{align*}
				|\eta s|\le e_f|\xi|^{\theta_2}+g_f|s|^{\theta_3}+d_f(x)
			\end{align*}
			for all $\eta\in f(x,s,\xi)$, for all $s\in \R$, for all $\xi\in\R^N$ and for a.\,a.\,$x\in \Omega$, where $\theta_1>0$ is given in H($j$)(iv) (see below),  $\lambda_{1,p}^S$ is the first eigenvalue of the $p$-Laplacian with Steklov boundary condition, see \eqref{Steklovbc} and \eqref{Steklovbc2}, and $\delta\colon[1,p]\to  [0,1]$ is defined by
			\begin{align*}
				\delta(\theta):=
				\begin{cases}
					0,&\text{if }\theta\in[1,p)\\
					1,&\text{if }\theta=p
				\end{cases}
				\quad \text{for all }\theta\in[1,p].
			\end{align*}
	\end{enumerate}
\end{enumerate}

\begin{enumerate}
	\item[H($\Phi$):]
		The function $\Phi\colon \Omega\to [0,\infty)$ is such that $\Phi\in L^{q_1'}(\Omega)$.
\end{enumerate}

\begin{enumerate}
	\item[H($j$):]
	The function $j\colon \Gamma_2\times \R\to \R$ satisfies the following conditions:
	\begin{enumerate}
		\item[(i)]
		$x\mapsto j(x,r)$ is measurable on $\Gamma_2$ for all $r\in\R$ such that  the function $x\mapsto j(x,0)$ belongs to $L^1(\Gamma_2)$;
		\item[(ii)]
		for  a.\,a.\,$x\in \Gamma_2$, $r\mapsto j(x,r)$ is locally Lipschitz continuous;
		\item[(iii)]
		there exist $a_j>0$ and $b_j\in L^{q_2'}(\Gamma_2)_+$ such that
		\begin{align*}
			|\xi|\le a_j|s|^{q_2-1}+b_j(x)
		\end{align*}
		for all $\xi\in \partial j(x,s)$, for all $s\in \R$ and for a.\,a.\,$x\in\Gamma_2$, where $1<q_2< p_*$ and $p_*$ is the critical exponent on the boundary given in \eqref{defp1*} for $r=p$;
		\item[(iv)]
		there exist $c_j>0$, $d_j\in L^1(\Gamma_2)_+$ and $\theta_1\in[1,p]$ such that
		\begin{align*}
			|\xi s|\le c_j|s|^{\theta_1}+d_j(x)
		\end{align*}
		for all $\xi\in \partial j(x,s)$, for all $s\in \R$ and for a.\,a.\,$x\in\Gamma_2$.
	\end{enumerate}
\end{enumerate}

Next, we give two concrete examples for functions $f$ and $j$ that satisfy hypotheses H($f$) and H($j$), respectively.
\begin{example}
	Let us consider the functions $f\colon \R\times \R^N\to 2^\R$ and $j\colon \R\to \R$ defined by
	\begin{align*}
		f(s,\xi)=|s|^\frac{p-1}{2}[-1,1]+|\xi|^\frac{p-1}{2},
	\end{align*}
	for all $s\in\R$, for all $\xi \in \R^N$ and
	\begin{align*}
		j(s)=
		\begin{cases}
			|s|&\text{if }s\in[-1,1],\\
			2-s&\text{if }s\in(1,2],\\
			s+2&\text{if }s\in[-2,-1),\\
			(|s|-2)^p&\text{if }|s|>2,
		\end{cases}
	\end{align*}
	for all $s\in\R$. It is not difficult to see that the functions $f$ and $j$ defined above satisfy  hypotheses \textnormal{H($f$)} and \textnormal{H($j$)}, respectively.
\end{example}

Let $K$ be a subset of $V$ defined by
\begin{equation}\label{defK}
	K:=\left\{u\in V\ \big | \ u(x)\le \Phi(x)\text{ for a.\,a.\,}x\in \Omega\right\}.
\end{equation}

\begin{remark}\label{remarks3.1}
	From hypothesis {\rm H($\Phi$)} it follows that $0\in K $ and $K$ is a nonempty, closed and convex subset of $V$.
\end{remark}

We are now in a position to give the following definition of weak solutions to  problems \eqref{eqn1}.

\begin{definition}
	A function $u\in K$ is called a weak solution of problem \eqref{eqn1} if there exists a function $\eta\in L^{q_1'}(\Omega)$ such that $\eta(x)\in f(x,u(x),\nabla u(x))$ for a.\,a.\,$x\in \Omega$ and
	\begin{align*}
		&\int_\Omega \left(|\nabla u|^{p-2} \nabla u \cdot \nabla (v-u)+\mu(x)|\nabla u|^{q-2} \nabla u\cdot \nabla (v- u)\right)\,\mathrm{d}x\\
&+\int_\Omega\left(|u|^{p-2}u+\mu(x)|u|^{q-2}u\right)(v-u)\,\mathrm{d}x+\int_{\Gamma_2}j^\circ(x,u;v-u)\,\mathrm{d}\Gamma\\
		&\ge \int_\Omega \eta(x) (v-u)\,\mathrm{d}x
	\end{align*}
	for all $v\in K$, where $K$ is defined in \eqref{defK}.
\end{definition}

\begin{remark}
	Note that the definition above is indeed equivalent to the usual one, see, for example, Giannessi \& Khan \cite[Proposition 3.3]{Giannessi-Khan-2000}.
\end{remark}

The main result in this section concerning the nonemptiness, boundedness and closedness of the solution set to problem \eqref{eqn1} is given by the following theorem.

\begin{theorem}\label{theorems3.3}
	Let hypotheses \eqref{conditions-p-q-mu}, \textnormal{H($f$)}, \textnormal{H($\Phi$)} and \textnormal{H($j$)} be satisfied. Then, the solution set $\mathcal{S}$ of problem \eqref{eqn1} is nonempty, bounded and weakly closed in $V$.
\end{theorem}

\begin{proof}
	{\bf I: Existence.} Let $\gamma\colon V\to X:=L^p(\Gamma_2)$ be the trace operator from $V$ into $X$. It is obvious from Proposition \ref{proposition_embeddings}(iii) that $\gamma$ is linear, continuous and compact. By $i\colon V\to L^{q_1}(\Omega)$ and $i^*\colon L^{q_1'}(\Omega)\to V^*$, we denote the embedding operator from $V$ to $L^{q_1}(\Omega)$ and its adjoint operator, respectively. Since $1<q_1<p^*$ we know that $i$ and $i^*$ are both compact operators by Proposition \ref{proposition_embeddings}\textnormal{(ii)}. Furthermore, we introduce the functional $J\colon X\to \R$ defined by
	\begin{align*}
		J(w):=\int_{\Gamma_2}j(x,w)\,\mathrm{d}\Gamma\quad \text{for all }w\in X.
	\end{align*}
	From hypotheses H($j$)(i)--(ii) and Theorem 3.47 of Mig\'{o}rski, Ochal \& Sofonea \cite{Migorski-Ochal-Sofonea-2013}, it is easy to see that $J$ is locally Lipschitz continuous and satisfies
	\begin{align}\label{eqns3.03}
		\begin{split}
			j^\circ(u;v)\le \int_{\Gamma_2}j^\circ(x,u(x);v(x))\,\mathrm{d}\Gamma
			\quad\text{and}\quad
			\partial J(u)\subset\int_{\Gamma_2}\partial j(x,u(x))\,\mathrm{d}\Gamma
		\end{split}
	\end{align}
	for all $u,v\in X$.

	For any $w\in X$ and any $\xi\in \partial J(w)$, it holds $\xi(x)\in \partial j(x,u(x))$ for a.\,a.\,$x\in \Gamma_2$. This along with hypothesis H($j$)(iii) deduces that
	\begin{align}\label{eqns3.3}
		\begin{split}
			\|\xi\|_{q_2',\Gamma_2}^{q_2'}
			&\le\int_{\Gamma_2}|\xi(x)|^{q_2'}\,\mathrm{d}\Gamma
			\le \int_{\Gamma_2}\l(a_j|w(x)|^{q_2-1}+b_j(x)\r)^{q_2'}\,\mathrm{d}\Gamma\\
			&\le M_0\int_{\Gamma_2}\l(|w(x)|^{q_2}+|b_j(x)|^{q_2'}\r)\,\mathrm{d}\Gamma= M_0\l(\|w\|_{q_2,\Gamma_2}^{q_2}+\|b_j\|_{q_2',\Gamma_2}^{q_2'}\r),
		\end{split}	
	\end{align}
	for some $M_0>0$. Hypotheses H($f$)(i), (iii) and the proof of Proposition 3 in Papageorgiou, Vetro \& Vetro \cite{Papageorgiou-Vetro-Vetro-2018} allow us to consider the Nemytskij operator $\mathcal N_f\colon V\subset L^{q_1}(\Omega)\to 2^{L^{q_1'}(\Omega)}$ corresponding to the multivalued mapping $f$ defined
	\begin{align*}
		\mathcal N_f(u):=\left \{\eta\in L^{q_1'}(\Omega)\ \big | \ \eta(x)\in f(x,u(x),\nabla u(x))\text{ for a.\,a.\,}x\in \Omega\right\}
	\end{align*}
	for all $u\in V$. Also, the growth condition in hypothesis H($f$)(iv) guarantees that
	\begin{align}\label{eqns3.4}
		\begin{split}
			\|\eta\|_{q_1',\Omega}^{q_1'}
			& =\int_\Omega|\eta(x)|^{q_1'}\,\mathrm{d}x\\
			& \le \int_\Omega \left(a_f|\nabla u|^{\frac{p}{q_1'}}+b_f|u|^{q_1-1}+c_f(x)\right)^{q_1'}\,\mathrm{d}x\\[1ex]
			& \le M_1\int_\Omega \l(|\nabla u|^p+|u|^{q_1}+c_f(x)^{q_1'}\r)\,\mathrm{d}x\\
			& =M_1\left(\|\nabla u\|_{p,\Omega}^p+\|u\|_{q_1,\Omega}^{q_1}+\|c_f\|_{q_1',\Omega}^{q_1'}\right),
		\end{split}
	\end{align}
	for some $M_1>0$. Keeping in mind that the embeddings $V \hookrightarrow  L^{q_1}(\Omega)$ and $V\hookrightarrow L^{q_2}(\Gamma_2)$ are both continuous (even compact), we see that $u\mapsto i^*\circ \mathcal N_f(u)+\gamma^*\partial J(\gamma u)$ maps bounded sets of $V$ into bounded sets of $V^*$.

	Finally, by $I_K\colon V\to \overline \R:=\R\cup\{+\infty\}$ we denote the indicator function of $K$, that is,
	\begin{align*}
		I_K(u):=
		\begin{cases}
			0&\text{if }u\in K,\\
			+\infty&\text{otherwise}.
		\end{cases}
	\end{align*}

	Let us now consider the following problem: find $u\in K$ such that
	\begin{align}\label{eqns3.6}
		Au-i^*\mathcal N_f(u)+\gamma^*\partial J(\gamma u)+\partial_cI_K(u)\ni 0\text{ in $V^*$},
	\end{align}
	where $A$ is given in (\ref{defA}) and $\partial_cI_K$ stands for the convex subdifferential operator of the convex function $I_K$ owning to the closedness and convexity of $K$, see Remark \ref{remarks3.1}.  By the definitions of the convex subgradient and generalized Clarke's subgradient along with inequality \eqref{eqns3.03}, we see that if $u\in K$ solves problem \eqref{eqns3.6}, then it is also a weak solution of problem \eqref{eqn1}. Based on this fact, we are going to verify the existence of a solution of problem \eqref{eqns3.6}.

	In order to obtain the existence of a solution to  problem \eqref{eqns3.6}, for each fixed $\varepsilon>0$, let us consider the following approximating problem associated to problem \eqref{eqns3.6}: find $u_\varepsilon\in V$ such that
	\begin{align}\label{eqns3.7}
		Au_\varepsilon-i^*\mathcal N_f(u_\varepsilon)+\gamma^*\partial J(\gamma u_\varepsilon)+I_{K,\varepsilon}'(u_\varepsilon)\ni 0\text{ in $V^*$},
	\end{align}
	where $I_{K,\varepsilon}\colon V\to \R$ is the Moreau-Yosida approximation of $I_K$ defined by
	\begin{equation*}
		I_{K,\varepsilon}(u):=\inf_{v\in V}\bigg(\frac{\|u-v\|_{V}^2}{2\varepsilon}+I_K(v)\bigg)
	\end{equation*}
	for all $u\in V$ and $I_{K,\varepsilon}'$ is the differential operator of $I_{K,\varepsilon}$, see Lemma \ref{lemma2}. First, we show that the multivalued mapping $u\mapsto Au-i^*\mathcal N_f(u)+\gamma^*\partial J(\gamma u)$ is pseudomonotone. Because of the boundedness, convexity and closedness of $u\mapsto i^*\circ \mathcal N_f(u)+\gamma^*\partial J(\gamma u)$, we conclude from Proposition~\ref{prop1} that  $u\mapsto Au-i^*\mathcal N_f(u)+\gamma^*\partial J(\gamma u)$ has nonempty, bounded, closed and convex values. Using \eqref{eqns3.3} and \eqref{eqns3.4} together with Proposition \ref{prop1}, it implies that $u\mapsto \mathcal Au:=Au-i^*\mathcal N_f(u)+\gamma^*\partial J(\gamma u)$ is a bounded mapping.

	Let $\{u_n\}_{n\in\N}\subset V$, $\{u_n^*\}_{n\in\N}\subset V^*$ and $u\in V$ be such that
	\begin{align}
		&u_n\wto u\quad \text{in }V,\quad u_n^*\wto u^* \quad\text{in }V^*,\label{eqns3.8}\\
		&u_n^*\in \mathcal A(u_n)\quad \text{for all }n\in\N,\nonumber\\
		&\limsup_{n\to\infty}\langle u_n^*,u_n-u\rangle \leq 0.\label{eqns3.10}
	\end{align}
	So, for each $n\in\mathbb N$, we are able to find  elements $\xi_n\in \mathcal N_f(u_n)$ and $\eta_n\in\partial J(\gamma u_n)$ such that $u_n^*=A(u_n)-i^*\xi_n+\gamma^*\eta_n$. Because of the compact embeddings $V\hookrightarrow L^{q_1}(\Omega)$ and $V\hookrightarrow L^{q_2}(\Gamma_2)$, we get from \eqref{eqns3.8} that $u_n\to u$ in $L^{q_1}(\Omega)$ and $u_n\to u$ in $L^{q_2}(\Gamma_2)$. Taking \eqref{eqns3.3} and \eqref{eqns3.4} into account, we can find bounded sequences $\{\xi_n\}_{n\in\N}\subset L^{q_1'}(\Omega)$ and $\{\eta_n\}_{n\in\N}\subset L^{q_2'}(\Gamma_2)$.  Therefore, without any loss of generality, we may assume that
	\begin{align}\label{covg1}
		\xi_n\wto \xi\quad \text{ in }L^{q_1'}(\Omega)
		\quad \text{and}\quad
		\eta_n\wto \eta\quad \text{in }L^{q_2'}(\Gamma_2)
	\end{align}
	for some $\xi\in L^{q_1'}(\Omega)$ and $\eta\in L^{q_2'}(\Gamma_2)$. Then, by \eqref{eqns3.10} and \eqref{covg1} we get
	\begin{align*}
		& \limsup_{n\to\infty}\langle A(u_n),u_n-u\rangle\\
		&= \limsup_{n\to\infty}\langle A(u_n),u_n-u\rangle -\lim_{n\to\infty}\langle \xi_n,u_n-u\rangle_{L^{q_1}(\Omega)}+\lim_{n\to\infty}\langle \eta_n,u_n-u\rangle_{L^{q_2}(\Gamma_2)}\\
		& =\limsup_{n\to\infty}\langle A(u_n)-i^*\xi_n+\gamma^*\eta_n,u_n-u\rangle\\
		& =\limsup_{n\to\infty}\langle u_n^*,u_n-u\rangle \leqslant 0.
	\end{align*}
	This fact together with \eqref{eqns3.8} and the $(\Ss_+)$-property of $A$ (see Proposition \ref{prop1}) implies that $u_n\to u$ in $V$. Recall that $\partial J$ is strongly-weakly closed, so we have $\eta\in \partial J(\gamma u)$, whereas from the continuity of $A$ (see Proposition~\ref{prop1}), we have
	\begin{align*}
		\langle u_n^*,u_n\rangle \to \langle u^*,u\rangle \quad \text{and}\quad A(u_n)\to A(u)
		\quad \text{in $V^*$}.
	\end{align*}
	
Since $\xi_n\in \mathcal N_f(u_n)$ it follows
	\begin{align*}
		\xi_n(x)\in f(x,u_n(x),\nabla u_n(x))\quad\textrm{for a.\,a.}\ x\in \Omega.
	\end{align*}
	From $u_n\to u$ in $V$ and the continuous embedding $V\hookrightarrow W^{1,p}(\Omega)$,  passing to a subsequence if necessary, we can assume that
	\begin{align*}
		u_n(x)\to u(x)\quad \text{and}\quad \nabla u_n(x)\to \nabla u(x) \quad  \text{for a.\,a.\,} x\in \Omega.
	\end{align*}
	Keeping in mind that $\R\times\R^N\ni(s,w)\mapsto f(x,s,w)\subset \R$ is upper semicontinuous and has nonempty closed convex values (see hypotheses H($f$)), we are now in a position to invoke Theorem 7.2.2 of Aubin \& Frankowska \cite[p. 273]{Aubin-Frankowska-1990} in order to conclude that
	\begin{align*}
		\xi(x)\in f(x,u(x),\nabla u(x))\quad \text{for a.\,a.\,} x\in \Omega.
	\end{align*}
	This means that $\xi\in \mathcal N_f(u)$. So we finally have that
	\begin{align*}
		u^*=A(u)-i^*\xi+\gamma^*\eta\in \mathcal A(u),
	\end{align*}
	which implies that $\mathcal A$ is generalized pseudomonotone. Note that $A$ is a bounded operator, it follows from Mig\'{o}rski, Ochal \& Sofonea \cite[Proposition 3.58(ii)]{Migorski-Ochal-Sofonea-2013} along with \eqref{eqns3.3} and \eqref{eqns3.4} that $\mathcal A$ is pseudomonotone.

	Employing Lemma \ref{lemma2} we infer that $I_{K,\varepsilon}'\colon V\to V^*$ is a bounded, demicontinuous and monotone operator, so $I_{K,\varepsilon}'$ is pseudomonotone as well. This allows us to apply Theorem 3.69 of Mig\'{o}rski, Ochal \& Sofonea \cite{Migorski-Ochal-Sofonea-2013} to get that $\mathcal A+I_{K,\varepsilon}'$ is also pseudomonotone.

	Furthermore, we are going to show that $u\mapsto \mathcal A(u)+I_{K,\varepsilon}'(u)$ is coercive. Let $u\in V$, $\xi\in \mathcal N_f(u)$ and $\eta\in \partial J(\gamma u)$ be arbitrary. The monotonicity of $I_{K,\varepsilon}'$ leads to
	\begin{align}\label{eqns3.14}
		\begin{split}
			&\langle A(u)-i^*\xi+\gamma^*\eta+I_{K,\varepsilon}'(u),u\rangle\\
			&= \int_\Omega|\nabla u|^{p-2} \nabla u \cdot \nabla u\,\mathrm{d}x+ \int_\Omega\mu(x)|\nabla u|^{q-2} \nabla u \cdot \nabla u\,\mathrm{d}x+\int_\Omega|u|^p+\mu(x)|u|^q\,\mathrm{d}x \\ &\quad -\int_\Omega\xi(x)u\,\mathrm{d}x+\int_{\Gamma_2}\eta(x)u\,\mathrm{d}\Gamma+\langle I_{K,\varepsilon}'(0),u\rangle
			+\langle I_{K,\varepsilon}'(u)- I_{K,\varepsilon}'(0),u\rangle\\
			&\geq \|\nabla u\|_{p,\Omega}^p+\|\nabla u\|_{q,\Omega,\mu}^q+\|u\|_{p,\Omega}^p+\|u\|_{q,\Omega,\mu}^q\\
			&\quad -\int_\Omega\xi(x)u\,\mathrm{d}x
			+\int_{\Gamma_2}\eta(x)u\,\mathrm{d}\Gamma+\langle I_{K,\varepsilon}'(0),u\rangle.
		\end{split}
	\end{align}
	Let $\varepsilon_1,\varepsilon_2,\varepsilon_3>0$ be arbitrary. Hypothesis H($j$)(iv) implies that
	\begin{align}\label{eqns3.15}
		\begin{split}
			\int_{\Gamma_2}|\eta(x)u|\,\mathrm{d}\Gamma
			&\le \int_{\Gamma_2}c_j|u|^{\theta_1}+d_j(x)\,\mathrm{d}\Gamma\\
			&\le
			\begin{cases}
				c_j\|u\|_{p,\Gamma_2}^p+\|d_j\|_{1,\Gamma_2}&\text{if }\theta_1=p,\\
				\varepsilon_1\|u\|_{p,\Gamma_2}^p+c_1(\varepsilon_1)+\|d_j\|_{1,\Gamma_2}&\text{if } \theta_1<p,
			\end{cases}\\
	&\le
			\begin{cases}
				c_j\left(\lambda_{1,p}^S\right)^{-1}\left(\|\nabla u\|_{p,\Omega}^p+\|u\|_{p,\Omega}^p\right)+\|d_j\|_{1,\Gamma_2}&\text{if }\theta_1=p,\\
				\varepsilon_1\|u\|_{p,\Gamma_2}^p+c_1(\varepsilon_1)+\|d_j\|_{1,\Gamma_2}&\text{if } \theta_1<p,
			\end{cases}
		\end{split}
	\end{align}
	for some $c_1(\varepsilon_1)>0$, where we have used Young's generalized inequality for the case $\theta_1<p$. Additionally, the growth condition H($f$)(iv) and the generalized Young inequality indicate that
	\begin{align}\label{eqns3.16}
		\begin{split}
			&\int_\Omega|\xi(x)u|\,\mathrm{d}x\\
			& \le\int_\Omega e_f|\nabla u|^{\theta_2}+g_f|u|^{\theta_3}+d_f(x)\,\mathrm{d}x\\
			&\le
			\begin{cases}
				e_f\|\nabla u\|_{p,\Omega}^p+g_f\|u\|_{p,\Omega}^p+\|d_f\|_{1,\Omega}&\text{if }\theta_2=\theta_3=p,\\
				e_f\|\nabla u\|_{p,\Omega}^p+\varepsilon_2\|u\|_{p,\Omega}^p+c_2(\varepsilon_2)+\|d_f\|_{1,\Omega}&\text{if }\theta_2=p \text{ and }\theta_3<p,\\
				\varepsilon_3\|\nabla u\|_{p,\Omega}^p+c_3(\varepsilon_3)+g_f\|u\|_{p,\Omega}^p+\|d_f\|_{1,\Omega}&\text{if }\theta_2<p \text{ and }\theta_3=p,\\
				\varepsilon_3\|\nabla u\|_{p,\Omega}^p+c_3(\varepsilon_3)+\varepsilon_2\|u\|_{p,\Omega}^p+c_2(\varepsilon_2)+\|d_f\|_{1,\Omega}&\text{if }\theta_2<p \text{ and }\theta_3<p,
			\end{cases}
		\end{split}
	\end{align}
	for some $c_2(\varepsilon_2),c_3(\varepsilon_3)>0$.  Let us choose $\varepsilon_1,\varepsilon_2,\varepsilon_3$ small enough. From \eqref{eqns3.14}, \eqref{eqns3.15}, \eqref{eqns3.16}, the continuity of the embedding $V\hookrightarrow W^{1,p}(\Omega)$ and the following estimate
	\begin{align*}
		\langle I_{K,\varepsilon}'(0),u\rangle\le \|I_{K,\varepsilon}'(0)\|_{V^*}\|u\|_V,
	\end{align*}
	it is not difficult to apply the inequality
	\begin{align*}
		\max\left\{e_f\delta(\theta_2))+c_j\left(\lambda_{1,p}^{S}\right)^{-1}, g_f\delta(\theta_3)+c_j\left(\lambda_{1,p}^{S}\right)^{-1}\delta(\theta_1)\right\}<1,
	\end{align*}
	in order to conclude that $u\mapsto \mathcal A(u)+I_{K,\varepsilon}'(u)$  is coercive on $V$.

	Now we are in a position to apply Theorem 3.74 Mig\'{o}rski, Ochal \& Sofonea \cite{Migorski-Ochal-Sofonea-2013} which yields that $u\mapsto \mathcal A(u)+I_{K,\varepsilon}'(u)$ is surjective. Therefore, for every $\varepsilon>0$, the inclusion \eqref{eqns3.7} has at least one solution $u_\varepsilon\in V$.

	Let $\{\varepsilon_n\}_{n\in\N}$ be a positive sequence such that $\varepsilon_n\to 0$ as $n\to\infty$ and let $u_n:=u_{\varepsilon_n}$ be a solution of problem \eqref{eqns3.7} corresponding to $\varepsilon=\varepsilon_n$ for every $n\in\N$. We claim that the sequence $\{u_n\}_{n\in\N}$ is bounded in $V$. Arguing indirectly, suppose that $\|u_n\|_V\to +\infty$ as $n\to\infty$. Since $u_n$ is a solution of problem \eqref{eqns3.7}, there exist $\xi_n\in \mathcal N_f(u_n)$ and $\eta_n\in\partial J(\gamma u_n)$ such that
	\begin{align*}
		\langle A(u_n)-i^*\xi_n+\gamma^*\eta_n+I_{K,\varepsilon_n}'(u_n),u_n\rangle= 0.
	\end{align*}
	Hence,
	\begin{align*}
		\langle A(u_n)-i^*\xi_n+\gamma^*\eta_n,u_n\rangle= \langle I_{K,\varepsilon_n}'(u_n),-u_n\rangle \le I_{K,\varepsilon_n}(0)-I_{K,\varepsilon_n}(u_n),
	\end{align*}
	where the last inequality is obtained by using the convexity of $I_{K,\varepsilon_n}$.  From the definitions of $I_K$ and $I_{K,\varepsilon_n}$, we can see that $0\le I_{K,\varepsilon}(w)\le I_{K}(w)$ for all $w\in V$. Since $0\in K$, we get $ I_{K,\varepsilon}(0)=0$. So, it holds
	\begin{align*}
		\langle A(u_n)-i^*\xi_n+\gamma^*\eta_n,u_n\rangle= \langle I_{K,\varepsilon_n}'(u_n),-u_n\rangle \le 0.
	\end{align*}
	Using the inequality above along with \eqref{eqns3.15} and \eqref{eqns3.16}, it is not difficult to prove that
	\begin{align*}
		0\ge \langle A(u_n)-i^*\xi_n+\gamma^*\eta_n,u_n\rangle\to +\infty.
	\end{align*}
	But this is a contradiction. Thus, the sequence $\{u_n\}_{n\in\N}$ is bounded in $V$. Passing to a subsequence if necessary, we may assume that
	\begin{align}\label{eqns3.18}
		u_n\wto u\quad \text{in }V
	\end{align}
	for some $u\in V$.

	For any $v\in V$ we have
	\begin{align}\label{eqns3.19}
		\begin{split}
			0
			&=\langle A(u_n)-i^*\xi_n+\gamma^*\eta_n+I_{K,\varepsilon_n}'(u_n),v-u_n\rangle\\
			&\le \langle A(u_n)-i^*\xi_n+\gamma^*\eta_n,v-u_n\rangle+I_{K,\varepsilon_n}(v)-I_{K,\varepsilon_n}(u_n),
		\end{split}
	\end{align}
	where the last inequality uses the convexity of $I_{K,\varepsilon}$. Employing \eqref{eqns3.18} and Lemma \ref{lemma2}(iii) yields
	\begin{align}\label{eqns3.20}
		\limsup_{n\to\infty}I_{K,\varepsilon_n}(v)\le I_{K}(v)
		\quad\text{and}\quad
		\liminf_{n\to\infty}I_{K,\varepsilon_n}(u_n)\ge I_{K}(u).
	\end{align}
	Using \eqref{eqns3.3} and \eqref{eqns3.4} again, we conclude that the sequences $\{\xi_n\}_{n\in\N}$ and $\{\eta_n\}_{n\in\N}$ are bounded in $L^{q_1'}(\Omega)$ and $L^{q_2'}(\Gamma_2)$, respectively. So, we may assume that
	\begin{align}\label{eqns3.21}
		\xi_n\wto \xi\quad\text{ in }L^{q_1'}(\Omega)
		\quad\text{and}\quad
		\eta_n\wto \eta \quad \text{in }L^{q_2'}(\Gamma_2),
	\end{align}
	for some $\xi\in L^{q_1'}(\Omega)$ and $\eta\in L^{q_2'}(\Gamma_2)$. However, the compactness of the embeddings $V\hookrightarrow L^{q_1}(\Omega)$ and $V\hookrightarrow L^{q_2}(\Gamma_2)$ along with the  convergence in \eqref{eqns3.18} deduce that $u_n\to u$ in $L^{q_1}(\Omega)$ and $L^{q_2}(\Gamma_2)$. This gives
	\begin{align}\label{eqns3.22}
		\langle -i^*\xi_n+\gamma^*\eta_n,u-u_n\rangle=0.
	\end{align}
	Setting $v= u$ in \eqref{eqns3.19} and passing to the upper limit as $n\to\infty$, by applying \eqref{eqns3.20} with $v=u$ and \eqref{eqns3.22} we obtain
	\begin{align*}
		&\limsup_{n\to\infty}\langle A(u_n),u_n-u\rangle \\
		&=\limsup_{n\to\infty}\left[\langle -i^*\xi_n+\gamma^*\eta_n,u-u_n\rangle+I_{K,\varepsilon_n}(u)-I_{K,\varepsilon_n}(u_n)\right]\\
		&\le \limsup_{n\to\infty}\langle -i^*\xi_n+\gamma^*\eta_n,u-u_n\rangle+\limsup_{n\to\infty}I_{K,\varepsilon_n}(u)-\liminf_{n\to\infty}I_{K,\varepsilon_n}(u_n)
		\le 0.
	\end{align*}
	The latter combined with \eqref{eqns3.18} and the $(\Ss_+)$-property of $A$ (see Proposition \ref{prop1}), confesses that $u_n\to u$ in $V$.  As we have done before, it can be verified that $\xi\in \mathcal N_f(u)$ and $\eta\in \partial J(\gamma u)$.

	Passing to the upper limit as $n\to\infty$ in \eqref{eqns3.19} by using \eqref{eqns3.20} and \eqref{eqns3.21} as well as the continuity of $A$, we have
	\begin{align*}
		0\le \langle A(u)-i^*\xi+\gamma^*\eta,v-u\rangle+I_{K}(v)-I_{K}(u)
	\end{align*}
	for all $v\in V$, where $\xi\in \mathcal N_f(u)$ and $\eta\in \partial J(\gamma u)$. This means that $u$ solves the inclusion problem \eqref{eqns3.6}. Consequently, $u\in K$ is also a weak solution of problem \eqref{eqn1}.\\
	
	{\bf II: Boundedness.} Arguing by contradiction, suppose that the solution set $\mathcal S$ of problem \eqref{eqn1} is unbounded. Then, we are able to find a sequence $\{u_n\}_{n\in\N}\subset \mathcal S$ such that $\|u_n\|_V\to \infty$ as $n\to\infty$. Arguing as above, for each $n\in\mathbb N$, we have $\xi_n\in L^{q_1'}(\Omega)$ with $\xi_n(x)\in f(x,u_n(x),\nabla u_n(x))$ for a.\,a.\,$x\in \Omega$ and
	\begin{align*}
		&\int_\Omega|\nabla u_n|^{p-2} \nabla u_n \cdot \nabla u_n\,\mathrm{d}x+ \int_\Omega\mu(x)|\nabla u_n|^{q-2} \nabla u_n \cdot \nabla u_n\,\mathrm{d}x+\int_\Omega\l(|u_n|^p+\mu(x)|u_n|^q\r)\,\mathrm{d}x\\
		&\quad -\int_\Omega\xi_n(x)u_n\,\mathrm{d}x
		 -\int_{\Gamma_2}j^\circ(x,u_n;-u_n)\,\mathrm{d}\Gamma\le 0.
	\end{align*}
	Applying \eqref{eqns3.14}--\eqref{eqns3.16} leads to
	\begin{align*}
		0 \geq &\int_\Omega|\nabla u_n|^{p-2} \nabla u_n \cdot \nabla u_n\,\mathrm{d}x+ \int_\Omega\mu(x)|\nabla u_n|^{q-2} \nabla u_n \cdot \nabla u_n\,\mathrm{d}x+\int_\Omega\l(|u_n|^p+\mu(x)|u_n|^q\r)\,\mathrm{d}x\\
		& -\int_\Omega\xi_n(x)u_n\,\mathrm{d}x-\int_{\Gamma_2}j^\circ(x,u_n;-u_n)\,\mathrm{d}\Gamma \to +\infty,
	\end{align*}
	which is a contradiction. Therefore,  the  solution set $\mathcal S$ of problem \eqref{eqn1} is bounded.\\
	
	{\bf III: Closedness.} Let $\{u_n\}_{n\in\N}\subset \mathcal S$ be any sequence such that $u_n\wto u$ in $V$. Then, for each $n\in\mathbb N$,  there exists $\eta_n\in L^{q_1'}(\Omega)$ such that $\eta_n(x)\in f(x,u_n(x),\nabla u_n(x))$ for a.\,a.\,$x\in \Omega$ and
	\begin{align}\label{eqns3.23}
		\begin{split}
			&\int_\Omega \left(|\nabla u_n|^{p-2} \nabla u_n \cdot \nabla (v-u_n)+\mu(x)|\nabla u_n|^{q-2} \nabla u_n\cdot \nabla (v- u_n)\right)\,\mathrm{d}x\\
			&+\int_\Omega\left(|u_n|^{p-2}u_n+\mu(x)|u_n|^{q-2}u_n\right)(v-u_n)\,\mathrm{d}x+\int_{\Gamma_2}j^\circ(x,u_n;v-u_n)\,\mathrm{d}\Gamma\\
&\ge \int_\Omega \eta_n(x) (v-u_n)\,\mathrm{d}x
		\end{split}
	\end{align}
	for all $v\in K$. The convexity and the closedness of $K$ ensures that $u\in K$. Recall that the embeddings $V \hookrightarrow L^{q_1}(\Omega)$ and $V\hookrightarrow L^{q_2}(\Gamma_2)$ are both compact and $\{\eta_n\}_{n\in\N}$ is bounded in $L^{q_1'}(\Omega)$ (see hypothesis H($f$)(iii)). Therefore, we have
	\begin{align*}
		\limsup_{n\to\infty}\int_{\Gamma_2}j^\circ(x,u_n;v-u_n)\,\mathrm{d}\Gamma
		&\le\int_{\Gamma_2}\limsup_{n\to\infty}j^\circ(x,u_n;v-u_n)\,\mathrm{d}\Gamma\\
		&\le \int_{\Gamma_2}j^\circ(x,u;v-u)\,\mathrm{d}\Gamma,
	\end{align*}
	and
	\begin{align*}
		&\lim_{n\to\infty}\int_\Omega \eta_n(x) (u-u_n)\,\mathrm{d}x=0,
	\end{align*}
	where we have used Fatou's Lemma and the upper semicontinuity of $(s,t)\mapsto j^\circ(x,s;t)$. Taking $v=u$ in \eqref{eqns3.23} and passing to the upper limit as $n\to \infty$ for the resulting inequality, we get
	\begin{align*}
		& \limsup_{n\to\infty}\langle A(u_n),u_n-u \rangle
		\le \limsup_{n\to\infty}\int_{\Gamma_2}j^\circ(x,u_n;v-u_n)\,\mathrm{d}\Gamma-\lim_{n\to\infty}\int_\Omega \eta_n(x) (u-u_n)\,\mathrm{d}x\le 0.
	\end{align*}
	This together with the convergence $u_n\wto u$ in $V$ and the $(\Ss_+)$-property of $A$ (see Proposition \ref{prop1}) deduces that $u_n\to u$ in $V$.

	From hypotheses H($f$) and boundedness of $\{\eta_n\}_{n\in\N}$, we can show that $\eta_n\wto \eta$ in $L^{q_1'}(\Omega)$ with some $\eta\in L^{q_1'}(\Omega)$ such that $\eta(x)\in f(x,u(x),\nabla u(x))$ for a.\,a.\,$x\in \Omega$. Taking the upper limit in inequality \eqref{eqns3.23} yields
	\begin{align*}
		&\int_\Omega \left(|\nabla u|^{p-2} \nabla u \cdot \nabla (v-u)+\mu(x)|\nabla u|^{q-2} \nabla u\cdot \nabla (v- u)\right)\,\mathrm{d}x\nonumber\\
		&+\int_\Omega\left(|u|^{p-2}u+\mu(x)|u|^{q-2}u\right)(v-u)\,\mathrm{d}x+\int_{\Gamma_2}j^\circ(x,u;v-u)\,\mathrm{d}\Gamma\\
&\ge \int_\Omega \eta(x) (v-u)\,\mathrm{d}x
	\end{align*}
	for all $v\in K$ with $\eta(x)\in f(x,u(x),\nabla u(x))$ for a.\,a.\,$x\in \Omega$. Thus, $u\in \mathcal S$, namely, $\mathcal S$ is weakly closed in $V$.
\end{proof}

Particularly, if $\Phi(x)\equiv+\infty$, then we have $K=V$. In this situation, we could carry out the same arguments as in the proof of Theorem \ref{theorems3.3} to get the following result.
\begin{theorem}
	Let hypotheses  \eqref{conditions-p-q-mu}, \textnormal{H($f$)} and \textnormal{H($\Phi$)} be satisfied. Then, the solution set of the elliptic inclusion
	\begin{equation*}
		\begin{aligned}
			D_\mu(u)+|u|^{p-2}u+\mu(x)|u|^{q-2}u &\in f(x,u,\nabla u)\quad && \text{\textnormal{in} } \Omega,\\
			u  &= 0 &&\text{\textnormal{on} } \Gamma_1,\\
			\frac{\partial u}{\partial \nu_\mu}&\in -\partial j(x,u)&&\text{\textnormal{on} }\Gamma_2,
		\end{aligned}
	\end{equation*}
	is nonempty, bounded and weakly closed in $V$.
\end{theorem}

Furthermore, when meas($\Gamma_2$)$=0$, that is, $\Gamma=\Gamma_1$, then we have the following corollary.
\begin{corollary}\label{corollarys1}
	Let hypotheses \eqref{conditions-p-q-mu}, \textnormal{H($f$)} and \textnormal{H($j$)} be satisfied. Then, the solution set of problem (\ref{eqn1.0}) is nonempty, bounded and weakly closed in $W^{1,\mathcal H}_0(\Omega)$, where $W^{1,\mathcal H}_0(\Omega)$ is the completion of $C_0^\infty(\Omega)$ in $W^{1,\mathcal H}(\Omega)$, namely,
	\begin{align*}
		W_0^{1,\mathcal H}(\Omega)=\overline {C^\infty_0(\Omega)}^{W^{1,\mathcal H}(\Omega)}.
	\end{align*}
\end{corollary}

\begin{remark}
	We point out that Corollary \ref{corollarys1} coincides with Theorem 3.3. of Zeng, Gasi\'{n}ski, Winkert \& Bai \cite{Zeng-Gasinski-Winkert-Bai-2021}. However, the proof of \cite[Theorem 3.3]{Zeng-Gasinski-Winkert-Bai-2021} is different from ours, since it is based on a surjectivity result of Le \cite{Le-2011} for multivalued mappings generated by the sum of a maximal monotone multivalued operator and a bounded multivalued pseudomonotone mapping. Additionally, hypotheses \eqref{conditions-p-q-mu} and {\rm H($f$)(iv)} are weaker than the ones used in \cite{Zeng-Gasinski-Winkert-Bai-2021}.
\end{remark}

\section{Convergence analysis}\label{Section4}
The problems with different constraints (for example, problems with obstacle effect) lead to various difficulties in numerical analysis and the study of regularity of solutions. In order to bypass and overcome these difficulties, some appropriate and useful approximating methods have been introduced and developed. Among these, penalty method as a  powerful and useful approach has been widely applied to study problems with obstacle constraints.  Based on this motivation, the section is devoted to apply a penalty for introducing a family of elliptic approximating problems without obstacle constraints associated with problem \eqref{eqn1}, and to establish a critical convergence theorem which shows that the solution set $\mathcal S$ can be approximated by the solution sets of approximating problems, denoted by $\{\mathcal S_n\}_{n\in\N}$, in the sense of Kuratowski. This convergence theorem will help for numerical analysis and stability research to double phase problems with obstacle constraints.

In what follows, we assume that $\{\rho_n\}_{n\in\N}$ is a sequence with $\rho_n>0$ for each $n\in\mathbb N$ such that $\rho_n\to 0$. Let us introduce a penalty operator $B\colon L^{q_1}(\Omega)\to L^{q_1'}(\Omega)$ associated to the set $K$ defined by
\begin{align}\label{defB}
	\langle Bu,v\rangle_{L^{q_1}(\Omega)}
	=\int_\Omega\left(u-\Phi\right)^+v\,\mathrm{d}x\quad \text{for all }u,v\in L^{q_1}(\Omega).
\end{align}

The following lemma gives some important properties of $B$, see Zeng, Bai, Gasi\'{n}ski \& Winkert \cite[Lemma 3.3]{Zeng-Bai-Gasinski-Winkert-2021}.
\begin{lemma}\label{lemmas4.1}
	If hypothesis {\rm H($\Phi$)} holds, then the function $B\colon L^{q_1}(\Omega)\to L^{q_1'}(\Omega)$ given in \eqref{defB} is bounded, demicontinuous and monotone.
\end{lemma}

For each $n\in\mathbb N$, we consider the following approximating problem corresponding to problem \eqref{eqn1}:
\begin{equation}\label{eqns4.1}
	\begin{aligned}
		D_\mu(u)+|u|^{p-2}u+\mu(x)|u|^{q-2}u+\frac{1}{\rho_n}\left(u-\Phi\right)^+  &\in f(x,u,\nabla u)&& \text{in } \Omega,\\
		u  &= 0 &&\text{on } \Gamma_1,\\
		\frac{\partial u}{\partial \nu_\mu} & \in -\partial j(x,u)&&\text{on }\Gamma_2.
	\end{aligned}
\end{equation}

The weak solutions of problem \eqref{eqns4.1} are understood as follows.
\begin{definition}
	A function $u\in V$ is called a weak solution of problem \eqref{eqns4.1} if there exists $\eta\in L^{q_1'}(\Omega)$ such that $\eta(x)\in f(x,u(x),\nabla u(x))$ for a.\,a.\,$x\in \Omega$ and
	\begin{align*}
		&\int_\Omega \left(|\nabla u|^{p-2} \nabla u \cdot \nabla (v-u)+\mu(x)|\nabla u|^{q-2} \nabla u\cdot \nabla (v- u)\right)\,\mathrm{d}x+\frac{1}{\rho_n}\int_\Omega\left(u-\Phi\right)^+(v-u)\,\mathrm{d}x\\
		&\quad+\int_\Omega\left(|u|^{p-2}u+\mu(x)|u|^{q-2}u\right)(v-u)\,\mathrm{d}x+\int_{\Gamma_2}j^\circ(x,u;v-u)\,\mathrm{d}\Gamma\\
		&\ge \int_\Omega \eta(x) (v-u)\,\mathrm{d}x
	\end{align*}
for all $v\in V$.
\end{definition}

The main results in the section concerning the existence and convergence properties to problem \eqref{eqns4.1} are given in the next theorem.

\begin{theorem}\label{main_theorem}
	Let hypotheses \eqref{conditions-p-q-mu}, \textnormal{H($f$)}, \textnormal{H($\Phi$)} and \textnormal{H($j$)} be satisfied. If $\{\rho_n\}_{n\in\N}$ is a sequence with $\rho_n>0$ for each $n\in\mathbb N$ such that $\rho_n\to 0$ as $n\to\infty$, then the following assertions hold true:
	\begin{enumerate}
		\item[{\rm (i)}]
			For each $n\in \mathbb N$, the set $\mathcal S_n$ of weak solutions to problem \eqref{eqns4.1} is nonempty, bounded and weakly closed in $V$.
	\item[{\rm (ii)}]
		It holds
		\begin{align*}
			\emptyset\neq w\text{-}\limsup\limits_{n\to\infty}\mathcal S_n=s\text{-}\limsup\limits_{n\to\infty}\mathcal S_n\subset \mathcal S.
		\end{align*}
	\item[{\rm (iii)}]
		For each $u\in s\text{-}\limsup\limits_{n\to\infty}\mathcal S_n$ and any sequence $\{\widetilde u_n\}_{n\in\N}$ with
		\begin{align*}
			\widetilde u_n\in \mathcal T(\mathcal S_n,u)\quad\text{for each }n\in\N,
		\end{align*}
		there exists a subsequence of $\{\widetilde u_n\}_{n\in\N}$ converging strongly to $u$ in $V$, where the set $ \mathcal T(\mathcal S_n,u)$ is defined by
		\begin{align*}
			\mathcal T(\mathcal S_n,u):=\big\{\widetilde u \in \mathcal S_n\,\mid\,\|u-\widetilde u\|_{V}\le\|u-v\|_{V} \text{ for all $ v\in \mathcal S_n$}\big\}.
		\end{align*}
	\end{enumerate}
\end{theorem}

\begin{proof}
	{\rm (i)} From Lemma \ref{lemmas4.1} we know that operator $B$ given in \eqref{defB} is continuous, monotone and satisfies the growth condition
	\begin{align*}
		\|Bu\|_{V^*}\le c_B\left(\|\Phi\|_{q_1',\Omega}+\|u\|_{q_1',\Omega}\right)
	\end{align*}
	for all $u\in V$ for some $c_B>0$. Arguing as in the proof of Theorem~\ref{theorems3.3}, we are able to prove that the solution set of problem \eqref{eqns4.1} is nonempty, bounded and weakly closed by considering $\mathcal N_f(\cdot)+B(\cdot)$ instead of $\mathcal N_f(\cdot)$.
	
	{\rm (ii)} The proof of this assertion is divided into three steps.

	{\bf Step 1.} The set $\bigcup\limits_{n\in \mathbb N}\mathcal S_n$ is uniformly bounded in $V$.

	Assume that  $\bigcup\limits_{n\in \mathbb N}\mathcal S_n$  is unbounded in $V$. Passing to a relabeled subsequence if necessary, we are able to find a sequence  $\{u_n\}_{n\in\N}\subset V$ with $u_n\in\mathcal S_n$ for each $n\in\mathbb N$ such that
	\begin{align*}
		\|u_n\|_{V}\to \infty\quad \text{as }n\to\infty.
	\end{align*}
	Then, for every $n\in \mathbb N$, there exists $\eta_n\in L^{q_1'}(\Omega)$ with $\eta_n(x)\in f(x,u_n(x),\nabla u_n(x))$ for a.\,a.\,$x\in \Omega$ such that
	\begin{align*}
		&\int_\Omega \left(|\nabla u_n|^{p-2} \nabla u_n +\mu(x)|\nabla u_n|^{q-2} \nabla u_n\right)\cdot \nabla v\,\mathrm{d}x
		+\frac{1}{\rho_n}\int_\Omega \left(u_n-\Phi\right)^+v\,\mathrm{d}x\\
		&\quad+\int_\Omega\left(|u_n|^{p-2}u_n+\mu(x)|u_n|^{q-2}u_n\right)v\,\mathrm{d}x+\int_{\Gamma_2}j^\circ(x,u_n;v)\,\mathrm{d}\Gamma\\
&\ge \int_\Omega \eta_n(x) v\,\mathrm{d}x
	\end{align*}
	for all $v\in V$. Taking $v=-u_n$ into the inequality above gives
	\begin{align*}
		&\int_\Omega \left(|\nabla u_n|^{p-2} \nabla u_n +\mu(x)|\nabla u_n|^{q-2} \nabla u_n\right)\cdot \nabla u_n\,\mathrm{d}x-\int_\Omega \eta_n(x) u_n\,\mathrm{d}x\\
		&\quad+\int_\Omega\left(|u_n|^{p-2}u_n+\mu(x)|u_n|^{q-2}u_n\right)u_n\,\mathrm{d}x-\int_{\Gamma_2}j^\circ(x,u_n;-u_n)\,\mathrm{d}\Gamma\\
&\le -\frac{1}{\rho_n}\int_\Omega \left(u_n-\Phi\right)^+u_n\,\mathrm{d}x.
	\end{align*}
	From the monotonicity of $B$ and the nonnegativity of $\Phi$ it follows that
	\begin{align*}
		&\int_\Omega \left(|\nabla u_n|^{p-2} \nabla u_n +\mu(x)|\nabla u_n|^{q-2} \nabla u_n\right)\cdot \nabla u_n\,\mathrm{d}x-\int_\Omega \eta_n(x) u_n\,\mathrm{d}x\\
		&\quad+\int_\Omega\left(|u_n|^{p-2}u_n+\mu(x)|u_n|^{q-2}u_n\right)u_n\,\mathrm{d}x-\int_{\Gamma_2}j^\circ(x,u_n;-u_n)\,\mathrm{d}\Gamma\\
&\le -\frac{1}{\rho_n}\int_\Omega\left[ \left(u_n-\Phi\right)^+-\left(0-\Phi\right)^+\right]u_n\,\mathrm{d}x\\
		& \le 0,
	\end{align*}
	that is,
	\begin{align}\label{eqnn16}
		&\|\nabla u_n\|_{p,\Omega}^p+\|\nabla u_n\|_{q,\Omega,\mu}^q+\|u_n\|_{p,\Omega}^p+\|u_n\|_{q,\Omega,\mu}^q-\int_\Omega \eta_n(x) u_n\,\mathrm{d}x\nonumber\\
&-\int_{\Gamma_2}j^\circ(x,u_n;-u_n)\,\mathrm{d}\Gamma\le 0.
	\end{align}
	Let $\varepsilon_1,\varepsilon_2,\varepsilon_3>0$ be arbitrary. From hypotheses H($j$)(iv) and H($f$)(iv), we have
	\begin{align}\label{eqnn16.0}
		\begin{split}
			\int_{\Gamma_2}j^\circ(x,u_n;-u_n)\,\mathrm{d}\Gamma
			& \le \int_{\Gamma_2}|\xi_n(x)u_n|\,\mathrm{d}\Gamma\\
			& \le \int_{\Gamma_2}c_j|u_n|^{\theta_1}+d_j(x)\,\mathrm{d}\Gamma\\
			&\le
			\begin{cases}
				c_j\|u_n\|_{p,\Gamma_2}^p+\|d_j\|_{1,\Gamma_2}&\text{if }\theta_1=p,\\
				\varepsilon_1\|u_n\|_{p,\Gamma_2}^p+c_1(\varepsilon_1)+\|d_j\|_{1,\Gamma_2}&\text{if }\theta_1<p,
			\end{cases}\\
&\le
			\begin{cases}
				c_j(\lambda_{1,p}^S)^{-1}\left(\|\nabla u_n\|_{p,\Omega}^p+\|u_n\|_{p,\Omega}^p\right)+\|d_j\|_{1,\Gamma_2}&\text{if }\theta_1=p,\\
				\varepsilon_1\|u_n\|_{p,\Gamma_2}^p+c_1(\varepsilon_1)+\|d_j\|_{1,\Gamma_2}&\text{if }\theta_1<p,
			\end{cases}
		\end{split}
	\end{align}
	and
	\begin{align}\label{eqnn17}
		\begin{split}
			&\int_\Omega\eta_n(x)u_n\,\mathrm{d}x\\
			&\le
			\begin{cases}
				e_f\|\nabla u_n\|_{p,\Omega}^p+g_f\|u_n\|_{p,\Omega}^p+\|d_f\|_{1,\Omega}&\text{if } \theta_2=\theta_3=p,\\
				e_f\|\nabla u_n\|_{p,\Omega}^p+\varepsilon_2\|u_n\|_{p,\Omega}^p+c_2(\varepsilon_2)+\|d_f\|_{1,\Omega}&\text{if } \theta_2=p \text{ and } \theta_3<p,\\
				\varepsilon_3\|\nabla u_n\|_{p,\Omega}^p+c_3(\varepsilon_3)+g_f\|u_n\|_{p,\Omega}^p+\|d_f\|_{1,\Omega}&\text{if } \theta_2<p \text{ and }\theta_3=p,\\
				\varepsilon_3\|\nabla u_n\|_{p,\Omega}^p+c_3(\varepsilon_3)+\varepsilon_2\|u_n\|_{p,\Omega}^p+c_2(\varepsilon_2)+\|d_f\|_{1,\Omega}&\text{if }\theta_2<p \text{ and }\theta_3<p,
			\end{cases}
		\end{split}
	\end{align}
	with some $c_1(\varepsilon_1),c_2(\varepsilon_2),c_3(\varepsilon_3)>0$, where $\xi_n\colon \Gamma_2\to\R$ is such that
	\begin{align*}
		\xi_n(x)(-u_n(x))=j^\circ(x,u_n(x);-u_n(x))\quad \text{for a.\,a.\,}x\in \Gamma_2.
	\end{align*}
	From \eqref{eqnn16}, \eqref{eqnn16.0}, \eqref{eqnn17} and the inequality $e_f\delta(\theta_2)+g_f\lambda_1\delta(\theta_3)+c_j\lambda_2\delta(\theta_1)<1$ along with the continuity of the embedding $V\hookrightarrow W^{1,p}(\Omega)$ we get
	\begin{align*}
				0&\ge \|\nabla u_n\|_{p,\Omega}^p+\|\nabla u_n\|_{q,\Omega,\mu}^q+\|u_n\|_{p,\Omega}^p+\|u_n\|_{q,\Omega,\mu}^q-\int_\Omega \eta_n(x) u_n\,\mathrm{d}x\nonumber\\
&\quad-\int_{\Gamma_2}j^\circ(x,u_n;-u_n)\,\mathrm{d}\Gamma\to \infty
	\end{align*}
	as $n\to\infty$. This is a contradiction. Therefore, the set $\bigcup\limits_{n\in \mathbb N}\mathcal S_n$ is uniformly bounded in $V$ and so Step 1 is verified.

	Let $\{u_n\}_{n\in\N}\subset V$ be a sequence such that $u_n\in \mathcal S_n$ for each $n\in\mathbb N$. By virtue of Step 1,  we may suppose that along a relabeled subsequence
	\begin{align}\label{eqnn19}
		u_n\wto u\quad\text{as }n\to\infty
	\end{align}
	for some $u\in V$. This means that the set $w\text{-}\limsup\limits_{n\to\infty} \mathcal S_n$ is nonempty.

	Next, we shall show that $w\text{-}\limsup\limits_{n\to\infty} \mathcal S_n$ is a subset of $\mathcal S$. For any $u\in w\text{-}\limsup\limits_{n\to\infty} \mathcal S_n$ fixed, passing to a subsequence if necessary, we are able to find  a sequence $\{u_n\}_{n\in\N}\subset V$ with $u_n\in \mathcal S_n$ for all $n\in\mathbb N$ such that \eqref{eqnn19} is satisfied. Our goal is to prove that $u\in \mathcal S$.

	{\bf Step 2.} $u(x)\le \Phi(x)$ for a.\,a. $x\in \Omega$.

	For every $n\in\mathbb N$, we have
	\begin{align}\label{eqnn20}
		&\frac{1}{\rho_n}\int_\Omega \left(u_n-\Phi\right)^+v\,\mathrm{d}x
		\le \langle Au_n,-v\rangle +\int_\Omega \eta_n(x)v\,\mathrm{d}x+\int_{\Gamma_2}j^\circ(x,u_n;-v)\,\mathrm{d}\Gamma.
	\end{align}
	Taking H\"older's inequality and hypothesis H($f$)(iii) into account yields
	\begin{align}\label{eqnn21}
		\int_\Omega \eta_n(x)v\,\mathrm{d}x \le M_2\left(\|\nabla 	u_n\|_{p,\Omega}^p+\|u_n\|_{q_1,\Omega}^{q_1}+\|c_f\|_{q_1',\Omega}^{q_1'}\right)^{\frac{1}{q_1'}}\|v\|_{q_1,\Omega}
	\end{align}
	for some $M_2>0$. From hypothesis H($j$)(iii) we obtain
	\begin{align}\label{eqnn21.0}
		\begin{split}
			& \int_{\Gamma_2}j^\circ(x,u_n;-v)\,\mathrm{d}\Gamma\le \int_{\Gamma_2} (a_j|u_n|^{q_2-1}+b_j(x))v\,\mathrm{d}\Gamma\\
			&\le M_3\left(\|u_n\|_{q_2,\Gamma_2}+\|b_j\|_{q_2',\Gamma_2}\right)\|v\|_{q_2,\Gamma_2}
		\end{split}
	\end{align}
	for some $M_3>0$. Putting \eqref{eqnn21} and \eqref{eqnn21.0} into \eqref{eqnn20},  by applying the boundedness of $A$ (see Proposition \ref{prop1}), the convergence \eqref{eqnn19} and the continuity of the embeddings $V \hookrightarrow W^{1,p}(\Omega)\hookrightarrow L^{q_1}(\Omega)$ and $V\hookrightarrow L^{q_*}(\Gamma_2)$, we find a constant $M_4>0$, which is independent of $n$, such that
	\begin{align*}
		\frac{1}{\rho_n}\int_\Omega \left(u_n-\Phi\right)^+v\,\mathrm{d}x \le M_4\|v\|_{V}.
	\end{align*}
	Hence,
	\begin{align*}
		\int_\Omega \left(u_n-\Phi\right)^+v\,\mathrm{d}x\le \rho_nM_4\|v\|_{V}
	\end{align*}
	for all $v\in V$. Letting  $n\to \infty$ in the inequality above, using the convergence \eqref{eqnn19}, the compactness of the embedding $V\hookrightarrow L^{q_1}(\Omega)$ and Lebesgue's Dominated Convergence Theorem, it gives
	\begin{align*}
		\int_\Omega \left(u-\Phi\right)^+v(x)\,\mathrm{d}x
		&= \int_\Omega \lim\limits_{n\to\infty}\left(u_n-\Phi\right)^+v\,\mathrm{d}x\\
		&= \lim\limits_{n\to\infty}\int_\Omega \left(u_n-\Phi\right)^+v(x)\,\mathrm{d}x\\
		&\le \lim\limits_{n\to\infty}\rho_nM_4\|v\|_{V}\\
		&=0
	\end{align*}
	for all $v\in V$. Therefore, we have $\left(u(x)-\Phi(x)\right)^+=0$ for a.\,a.\,$x\in \Omega$, thus, $u(x)\le \Phi(x)$ for a.\,a.\,$x\in \Omega$.
	
	{\bf Step 3.} $u\in \mathcal S$.

	Note that
	\begin{align*}
		\langle Au_n,u_n-v\rangle
		&\le \frac{1}{\rho_n}\int_\Omega \left(u_n-\Phi\right)^+(v-u_n)\,\mathrm{d}x+\int_\Omega \eta_n(x)(u_n-v)\,\mathrm{d}x\\
		&\quad +\int_{\Gamma_2}j^\circ(x,u_n;v-u_n)\,\mathrm{d}\Gamma
	\end{align*}
	for all $v\in V$. From the monotonicity of $s\mapsto s^+$ we obtain
	\begin{align*}
		\langle Au_n,u_n-v\rangle
		& \le \frac{1}{\rho_n}\int_\Omega \left(v-\Phi\right)^+(v-u_n)\,\mathrm{d}x+\int_\Omega \eta_n(x)(u_n-v)\,\mathrm{d}x\\
		&\quad+\int_{\Gamma_2}j^\circ(x,u_n;v-u_n)\,\mathrm{d}\Gamma
	\end{align*}
	for all $v\in V$. By virtue of the definition of $K$ (see  \eqref{defK}), we have
	\begin{align}\label{eqnn22}
		\langle Au_n,u_n-v\rangle-\int_\Omega \eta_n(x)(u_n-v)\,\mathrm{d}x-\int_{\Gamma_2}j^\circ(x,u_n;v-u_n)\,\mathrm{d}\Gamma\le 0
	\end{align}
	for all $v\in K$.

	From Step 2 we know that $u\in K$. So, taking $v=u$ in \eqref{eqnn22} leads to
	\begin{align*}
		\langle Au_n,u_n-u\rangle\le \int_\Omega \eta_n(x)(u_n-u)\,\mathrm{d}x+\int_{\Gamma_2}j^\circ(x,u_n;u-u_n)\,\mathrm{d}\Gamma.
	\end{align*}
	Keeping in mind that the embeddings $V\hookrightarrow L^{q_1}(\Omega)$ and $V\hookrightarrow L^{q_2}(\Gamma_2)$ are both compact, we have
	\begin{align*}
		\lim_{n\to\infty} \int_\Omega \eta_n(x)(u_n-u)\,\mathrm{d}x=0\quad \text{and}\quad\limsup_{n\to\infty}\int_{\Gamma_2}j^\circ(x,u_n;u-u_n)\,\mathrm{d}\Gamma \le 0,
	\end{align*}
	where we have used the boundedness of $\{\eta_n\}_{n\in\N}\subset L^{q_1'}(\Omega)$, upper semicontinuity of $(s,t)\mapsto j^\circ(x,s;t)$ and Fatou's Lemma.  So, it holds
	\begin{align*}
		\limsup\limits_{n\to\infty}\langle Au_n,u_n-u\rangle\le 0.
	\end{align*}
	This combined with the convergence \eqref{eqnn19} and the  $(\Ss_+)$-property of $A$ (see Proposition \ref{prop1}) concludes that $u_n\to u$. Therefore, we have $w\text{-}\limsup\limits_{n\to \infty}\mathcal S_n\subset s\text{-}\limsup\limits_{n\to \infty}\mathcal S_n$. This together with  $s\text{-}\limsup\limits_{n\to\infty}{\mathcal S}_n\subset w\text{-}\limsup\limits_{n\to\infty}{\mathcal S}_n$ implies that $\emptyset\neq w\text{-}\limsup\limits_{n\to \infty}\mathcal S_n= s\text{-}\limsup\limits_{n\to \infty}\mathcal S_n$.

	Since $\{\eta_n\}_{n\in\N}$ is bounded in $L^{q_1'}(\Omega)$, as done before, we can also prove that $\eta_n\wto \eta$ in $L^{q_1'}(\Omega)$ as $n\to\infty$ for some $\eta\in L^{q_1'}(\Omega)$ with $\eta(x)\in f(x,u(x),\nabla u(x))$ for a.\,a.\,$x\in \Omega$. Consequently, we conclude that $\emptyset\neq  w\text{-}\limsup\limits_{n\to \infty}\mathcal S_n=s\text{-}\limsup\limits_{n\to \infty}\mathcal S_n\subset \mathcal S$.

	{\rm (iii)} For any fixed $u\in s\text{-}\limsup\limits_{n\to\infty}\mathcal S_n$, the nonemptiness, boundedness and closedness of  $\mathcal S_n$ guarantees that the set $ \mathcal T(\mathcal S_n,u)$ is well-defined. Let $\{\widetilde u_n\}_{n\in\N}$ be  any sequence such that
	\begin{align*}
		\widetilde u_n\in  \mathcal T(\mathcal S_n,u)\quad\text{ for each }n\in\mathbb N.
	\end{align*}
	It follows from Step~1 that the sequence $\{\widetilde u_n\}_{n\in\N}$ is bounded. So, without any loss of generality, we may assume that
	\begin{align}\label{eqns4.11}
		\widetilde u_n\wto \widetilde u \quad \text{in }V \text{ as }n\to\infty
	\end{align}
	for some $\widetilde u\in V$. Arguing as in the proof of Step~2, we obtain that $\widetilde u\in K$. Then, for each $n\in\mathbb N$, we have
	\begin{align*}
		\langle A\widetilde u_n,\widetilde u_n-v\rangle
		&\le \frac{1}{\rho_n}\int_\Omega \left(\widetilde u_n-\Phi\right)^+(v-\widetilde u_n)\,\mathrm{d}x+\int_\Omega \eta_n(x)(\widetilde u_n-v)\,\mathrm{d}x\\
		&\quad +\int_{\Gamma_2}j^\circ(x,u_n;v-u_n)\,\mathrm{d}\Gamma
	\end{align*}
	for all $v\in V$. Proceeding in the same way as in the proof of Step 3, we conclude that $\widetilde u $ is a solution to problem \eqref{eqn1} as well. Since $u\in s\text{-}\limsup\limits_{n\to\infty}\mathcal S_n$, passing to a subsequence if necessary, there exists a sequence $\{u_n\}_{n\in\N}$ such that $u_n\in \mathcal S_n$ and $u_n\to u$ in $V$ as $n\to\infty$.  The latter combined with \eqref{eqns4.11} deduces that
	\begin{align*}
		\|\widetilde u-u\|_V\le \liminf_{n\to\infty}\|\widetilde u_n-u\|_V\le \liminf_{n\to\infty}\|u_n-u\|_V=0,
	\end{align*}
	this means that $\widetilde u=u$. Consequently, the desired conclusion is proved.
\end{proof}

If meas($\Gamma_2$)$=0$, namely $\Gamma_1=\Gamma$, then Theorem~\ref{main_theorem} reduces the following corollary, which coincides with~\cite[Theorem 3.4]{ Zeng-Bai-Gasinski-Winkert-2021}.
\begin{corollary}
	Let hypotheses \eqref{conditions-p-q-mu}, \textnormal{H($f$)} and \textnormal{H($\Phi$)} be satisfied. If $\{\rho_n\}_{n\in\N}$ is a sequence with $\rho_n>0$ for each $n\in\mathbb N$ such that $\rho_n\to 0$ as $n\to\infty$, then the following statements hold true:
	\begin{enumerate}
		\item[{\rm (i)}]
			For each $n\in \mathbb N$, the set $\tilde {\mathcal S_n}$ of weak solutions of the following  problem  is nonempty, bounded and weakly closed in $W_0^{1,\mathcal H}(\Omega)$
			\begin{equation*}
				\begin{aligned}
				D_\mu(u)+\frac{1}{\rho_n}\left(u-\Phi\right)^+  &\in f(x,u,\nabla u)&& \text{\textnormal{in} } \Omega,\\
				u  &= 0 &&\text{\textnormal{on} } \Gamma.
				\end{aligned}
			\end{equation*}
		\item[{\rm (ii)}]
			It holds
			\begin{align*}
				\emptyset\neq w\text{-}\limsup\limits_{n\to\infty}\tilde{\mathcal S_n}=s\text{-}\limsup\limits_{n\to\infty}\tilde{\mathcal S_n}\subset \tilde{\mathcal S},
			\end{align*}
			where $\tilde {\mathcal S}$ is the solution set to problem \eqref{eqn1.0}.
		\item[{\rm (iii)}]
			For each $u\in s\text{-}\limsup\limits_{n\to\infty}\tilde{\mathcal S_n}$ and any sequence $\{\widetilde u_n\}_{n\in\N}$ with
			\begin{align*}
				\tilde u_n\in \mathcal T(\tilde{\mathcal S_n},u)\quad\text{for each }n\in\mathbb N,
			\end{align*}
			there exists a subsequence of $\{\tilde u_n\}_{n\in\N}$ converging strongly to $u$ in $W_0^{1,\mathcal H}(\Omega)$, where the set $ \mathcal T(\tilde{\mathcal S_n},u)$ is defined by
			\begin{align*}
				\mathcal T(\tilde{\mathcal S_n},u):=\left\{\tilde u \in\tilde {\mathcal S_n}\,\mid\,\|u-\tilde u\|_{W_0^{1,\mathcal H}(\Omega)}\le\|u-v\|_{W_0^{1,\mathcal H}(\Omega)} \text{ for all $ v\in \tilde{\mathcal S_n}$}\right\}.
			\end{align*}
	\end{enumerate}
\end{corollary}

\section*{Acknowledgments}
	The authors wish to thank the two knowledgeable referees for their remarks in order to improve the paper.
	
	This project has received funding from the NNSF of China Grant Nos. 12001478, 12026255 and 12026256, and the European Union's Horizon 2020 Research and Innovation Programme under the Marie Sklodowska-Curie grant agreement No. 823731 CONMECH, National Science Center of Poland under Preludium Project No. 2017/25/N/ST1/00611,  the Startup Project of Doctor Scientific Research of Yulin Normal University No. G2020ZK07, and the Natural Science Foundation of Guangxi Grant Nos. 2020GXNSFBA297137 and 2018GXNSFDA138002.  The research of Vicen\c tiu D. R\u adulescu was supported by a grant of the Romanian Ministry of Research, Innovation and Digitization, CNCS/CCCDI--UEFISCDI, project number PCE 137/2021, within PNCDI III.

\end{document}